\documentclass[11pt]{article} 
\usepackage{amsmath,amssymb,amsthm}
\usepackage{fancybox}  
\usepackage{graphicx}
\usepackage{color}

\allowdisplaybreaks

\begin{document}

\def\fl#1{\left\lfloor#1\right\rfloor}
\def\cl#1{\left\lceil#1\right\rceil}
\def\ang#1{\left\langle#1\right\rangle}
\def\stf#1#2{\left[#1\atop#2\right]} 
\def\sts#1#2{\left\{#1\atop#2\right\}}
\def\eul#1#2{\left\langle#1\atop#2\right\rangle}
\def\N{\mathbb N}
\def\Z{\mathbb Z}
\def\R{\mathbb R}
\def\C{\mathbb C}
\newcommand{\ctext}[1]{\raise0.2ex\hbox{\textcircled{\scriptsize{#1}}}}

\newtheorem{theorem}{Theorem}
\newtheorem{Prop}{Proposition}
\newtheorem{Cor}{Corollary}
\newtheorem{Lem}{Lemma}

\newenvironment{Rem}{\begin{trivlist} \item[\hskip \labelsep{\it
Remark.}]\setlength{\parindent}{0pt}}{\end{trivlist}}

\title{The $p$-Frobenius and $p$-Sylvester numbers for Fibonacci and Lucas triplets
}
                       
\author{
Takao Komatsu 
\\
\small Department of Mathematical Sciences, School of Science\\[-0.8ex]
\small Zhejiang Sci-Tech University\\[-0.8ex]
\small Hangzhou 310018 China\\[-0.8ex]
\small \texttt{komatsu@zstu.edu.cn}\\\\
Haotian Ying
\\ 
\small Department of Mathematical Sciences, School of Science\\[-0.8ex]
\small Zhejiang Sci-Tech University\\[-0.8ex]
\small Hangzhou 310018 China\\[-0.8ex]
\small \texttt
{tomyinght@gmail.com}  
}

\date{
%\small Submitted: February 15, 2021;  Accepted: March 25, 2021.\\
\small MR Subject Classifications: Primary 11D07; Secondary 05A15, 11B39, 05A17, 05A19, 11D04, 11P81, 20M14 
}

\maketitle
 
\begin{abstract} 
In this paper we study a certain kind of generalized linear Diophantine problem of Frobenius. Let $a_1,a_2,\dots,a_l$ be positive integers such that their greatest common divisor is one. For a nonnegative integer $p$, denote the $p$-Frobenius number by $g_p(a_1,a_2,\dots,a_l)$, which is the largest integer that can be represented at most $p$ ways by a linear combination with nonnegative integer coefficients of $a_1,a_2,\dots,a_l$. When $p=0$, $0$-Frobenius number is the classical Frobenius number. When $l=2$, $p$-Frobenius number is explicitly given. However, when $l=3$ and even larger, even in special cases, it is not easy to give the Frobenius number explicitly, and it is even more difficult when $p>0$, and no specific example has been known. However, very recently, we have succeeded in giving explicit formulas for the case where the sequence is of triangular numbers \cite{Ko22a} or of repunits \cite{Ko22b} for the case where $l=3$. In this paper, we show the explicit formula for the Fibonacci triple when $p>0$. In addition, we give an explicit formula for the $p$-Sylvester number, that is, the total number of nonnegative integers that can be represented in at most $p$ ways. Furthermore, explicit formulas are shown concerning the Lucas triple. 
      
\end{abstract}

\section{Introduction}

The linear Diophantine problem of Frobenius is to find the largest integer which is not expressed by a nonnegative linear combination of given positive relatively prime integers $a_1,a_2,\dots,a_l$. Such a largest integer is called the {\it Frobenius number} \cite{sy1884}, denoted by $g(A)=g(a_1,a_2,\dots,a_l)$, where $A=\{a_1,a_2,\dots,a_l\}$. 
In the literature on the Frobenius problem, the {\it Sylvester number} or {\it genus} $n(A)=n(a_1,a_2,\dots,a_l)$, which is the total number of integers that cannot be represented as a nonnegative linear combination of $a_1,a_2,\dots,a_l$ \cite{sy1882}. 

%Other similar or extended numbers, including the sum of such numbers that cannot be represented as a nonnegative linear combination of $a_1,a_2,\dots,a_l$ \cite{tr08} and the weighted sum of such numbers \cite{KZ0,KZ}, have been also proposed and studied. This stream of concepts is clarified and generalized in \cite{Ko22} by using Bernoulli numbers. $p$-generalizations may be similarly defined. However, because of the space, we do not treat with them in this paper.  

There are many aspects studying the Frobenius problem. For example, there are algorithmic aspects to find the values or the bounds, complexity of computations, denumerants, numerical semigroup, applications to algebraic geometry and so on (see, e.g., \cite{ADG20,ra05}). 
Nevertheless, one of the motivations for our $p$-generalizations originats in the number of representations $d(n;a_1,a_2,\dots,a_l)$ to $a_1 x_1+a_2 x_2+\dots+a_l x_l=n$ for a given positive integer $n$. This number is equal to the coefficient of $x^n$ in $1/(1-x^{a_1})(1-x^{a_2})\cdots(1-x^{a_l})$ for positive integers $a_1,a_2,\dots,a_l$ with $\gcd(a_1,a_2,\dots,a_l)=1$ \cite{sy1882}.   
Sylvester \cite{sy1857} and Cayley \cite{cayley} showed that $d(n;a_1,a_2,\dots,a_l)$ can be expressed as the sum of a polynomial in $n$ of degree $k-1$ and a periodic function of period $a_1 a_2\cdots a_l$.  
In \cite{bgk01}, the explicit formula for the polynomial part is derived by using Bernoulli numbers. For two variables, a formula for $d(n;a_1,a_2)$ is obtained in \cite{tr00}. For three variables in the pairwise coprime case $d(n;a_1,a_2,a_3)$, in \cite{ko03}, the periodic function part is expressed in terms of trigonometric functions, and its results have been improved in \cite{bi20} by using floor functions so that three variables case can be easily worked with in the formula.  

In this paper, we are interested in one of the most general and most natural types of Frobenius numbers, which focuses on the number of representations. 
For a nonnegative integer $p$, the largest integer such that the number of expressions that can be represented by $a_1,a_2,\dots,a_l$ is at most $p$ is denoted by $g_p(A)=g_p(a_1,a_2,\dots,a_l)$ and may be called the {\it $p$-Frobenius number}. That is, all integers larger than $g_p(A)$ have at least the number of representations of $p+1$ or more. This generalized Frobenius number $g_p(A)$ is called the $p$-Frobenius number \cite{Ko22a,Ko22b}, which is also called the $k$-Frobenius number \cite{BDFHKMRSS} or the $s$-Frobenius number \cite{FS11}. When $p=0$, $g(A)=g_0(A)$ is the original Frobenius number.  
One can consider the largest integer $g_p^\ast(a_1,a_2,\dots,a_l)$ that has exactly $p$ distinct representations (see, e.g., \cite{BDFHKMRSS,FS11}). However, in this case, the ordering $g_0^\ast\le g_1^\ast\le \cdots$ may not hold. For example, $g_{17}^\ast(2,5,7)=43>g_{18}^\ast(2,5,7)=42$. In addition, for some $j$, $g_j^\ast$ may not exist. For example, $g_{22}^\ast(2,5,7)$ does not exist because there is no positive integer whose number of representations is exactly $22$. 
Therefore, in this paper we do not study $g_p^\ast(A)$ but $g_p(A)$. 

Similarly to the $p$-Frobenius number, the {\it $p$-Sylvester number} or the {\it $p$-genus} $n_p(A)=n_p(a_1,a_2,\dots,a_l)$ is defined by the cardinality of the set of integers which can be represented by $a_1,a_2,\dots,a_l$ at most $p$ ways. When $p=0$, $n(A)=n_0(A)$ is the original Sylvester number.  

In this paper, we are interested in one of the most crucial topics, that is, to find explicit formulas of indicators, in particular, of $p$-Frobenius numbers and $p$-Sylvester numbers. In the classical case, that is, for $p=0$, explicit formulas of $g(a_1,a_2)$ and $n(a_1,a_2)$ are shown when $l=2$ \cite{sy1882,sy1884}.   
However, for $l\ge 3$, $g(A)$ cannot be given by any set of closed formulas which can be reduced to a finite set of certain polynomials \cite{cu90}. For $l=3$, there are several useful algorithms to obtain the Frobenius number (see, e.g., \cite{Fel2006,Johnson,Rodseth}). For the concretely given three positive integers, if the conditions are met, the Frobenius number can be completely determined by the method of case-dividing by A. Tripathi \cite{tr17}.  
Although it is possible to find the Frobenius number by using the results in \cite{tr17}, it is another question whether the Frobenius number can be given by a closed explicit expression for some special triplets, and special considerations are required. 
 Only for some special cases, explicit closed formulas have been found, including arithmetic, geometric, Mersenne, repunits and triangular (see \cite{RR18,RBT2015,RBT2017} and references therein).  

For $p>0$, if $l=2$, explicit formulas of $g_p(a_1,a_2)$ and $n_p(a_1,a_2)$ are still given without any difficulty (see, e.g., \cite{bk11}). However, if $l\ge 3$, no explicit formula had been given even in a special case. However, quite recently, we have succeeded in giving explicit formulas for the case where the sequence is of triangular numbers \cite{Ko22a} or of repunits \cite{Ko22b} for the case where $l=3$. 

 In this paper, we give an explicit formula for the $p$-Frobenius number for the Fibonacci number triple $(F_i,F_{i+2},F_{i+k})$ ($i,k\ge 3$). Here, the $n$-th Fibonacci number $F_n$ is defined by $F_n=F_{n-1}+F_{n-2}$ ($n\ge 2$) with $F_1=1$ and $F_0=0$. Our main results (Theorem \ref{th:dp} below) is a kind of generalizations of \cite[Theorem 1]{MAR} when $p=0$. However, when $p>0$, the exact situation is not completely similar to the case where $p=0$, and the case by case discussion is necessary.  
As analogues, we also show explicit formulas of $g_p(L_i,L_{i+2},L_{i+k})$ for Lucas numbers $L_n$ with $i,k\ge 3$. Here, Lucas numbers $L_n$ satisfy the recurrence relation $L_n=L_{n-1}+L_{n-2}$ ($n\ge 2$) with $L_0=2$ and $L_1=1$. 
By using our constructed framework, we can also find explicit formulas of the $p$-Sylvester numbers $n_p(F_i,F_{i+2},F_{i+k})$ and $n_p(L_i,L_{i+2},L_{i+k})$. Our  result (Theorem \ref{th:np}) can extend the result in \cite[Corollary 2]{MAR}.    
The main idea is to find the explicit structure of the elements of an Ap\'ery set \cite{Apery}. In addition, we use a complete residue system, studied initially by Selmer \cite{se77}. By using Ap\'ery sets, we construct the first least set of the complete residue system, then the second least set of the complete residue system, and the third, and so on. As a basic framework, we use a similar structure in \cite{Ko22b}. We can safely say that one of our theorems (Theorem \ref{th:dp} below) is a kind of generalizations of \cite[Theorem 1]{MAR}. Nevertheless, for each nonnegative integer $p$, the exact situation is not completely similar, but the case by case discussion is necessary.

\section{Preliminaries}   

Without loss of generality, we assume that $a_1=\min\{a_1,a_2,\dots,a_l\}$.   
For each $0\le i\le a_1-1$, we introduce the positive integer $m_i^{(p)}$ congruent to $i$ modulo $a_1$ such that the number of representations of $m_i^{(p)}$ is bigger than or equal to $p+1$, and that of $m_i-a_1$ is less than or equal to $p$. Note that $m_0^{(0)}$ is defined to be $0$.      The set 
$$
{\rm Ap} (A;p)={\rm Ap} (a_1,a_2,\dots,a_l;p)=\{m_0^{(p)},m_1^{(p)},\dots,m_{a_1-1}^{(p)}\}\,, 
$$ 
is called the {\it $p$-Ap\'ery set} of $A=\{a_1,a_2,\dots,a_l\}$ for a nonnegative integer $p$, which is congruent to the set 
$$
\{0,1,\dots,a_1-1\}\pmod{a_1}\,. 
$$  
When $p=0$, the $0$-Ap\'ery set is the original Ap\'ery set \cite{Apery}. 

It is hard to find any explicit formula of $g_p(a_1,a_2,\dots,a_l)$ when $l\ge 3$.  Nevertheless, the following convenient formulas are known (see \cite{Ko22}). After finding the structure of $m_j^{(p)}$, we can obtain $p$-Frobenius or $p$-Sylvester numbers for triple $(F_i,F_{i+2},F_{i+k})$.   

\begin{Lem}  
Let $k$, $p$ and $\mu$ be integers with $k\ge 2$, $p\ge 0$ and $\mu\ge 1$.  
Assume that $\gcd(a_1,a_2,\dots,a_l)=1$.  We have 
\begin{align}  
g_p(a_1,a_2,\dots,a_l)&=\max_{0\le j\le a_1-1}m_j^{(p)}-a_1
\label{mp-g}\,,
\\  
n_p(a_1,a_2,\dots,a_l)&=\frac{1}{a_1}\sum_{j=0}^{a_1-1}m_j^{(p)}-\frac{a_1-1}{2}\,. 
\label{mp-n} 
\end{align} 
\label{cor-mp}
\end{Lem} 

\noindent 
{\it Remark.}  
When $p=0$, (\ref{mp-g}) is the formula by Brauer and Shockley \cite{bs62}: 
\begin{equation}   
g(a_1,a_2,\dots,a_l)=\left(\max_{1\le j\le a_1-1}m_j\right)-a_1\,, 
\label{eq:bs}
\end{equation} 
where $m_j=m_j^{(0)}$ ($1\le j\le a_1-1$) with $m_0=0$.  
When $p=0$, (\ref{mp-n}) is the formula by Selmer \cite{se77}: 
\begin{equation}   
n(a_1,a_2,\dots,a_l)=\frac{1}{a_1}\sum_{j=1}^{a_1-1}m_j-\frac{a_1-1}{2}\,.  
\label{eq:se}
\end{equation}
Note that $m_0=m_0^{(0)}=0$.  
A more general form by using Bernoulli numbers is given in \cite{Ko22}, as well as the concept of weighted sums \cite{KZ0,KZ}.   
\bigskip 

It is necessary to find the exact situation of $0$-Ap\'ery set ${\rm Ap}(F_i,0)$, the least complete residue system, which was initially studied in \cite{se77}. Concerning Fibonacci numbers, we use the framework in \cite{MAR}.   

Throughout this paper, for a fixed integer $i$, we write 
$$
{\rm Ap} (F_i;p)=\{m_0^{(p)},m_1^{(p)},\dots,m_{F_i-1}^{(p)}\}
$$ 
for short.  Then, we shall construct the set of the least complete residue system ${\rm Ap} (F_i;0)$. That is, $m_j\not\equiv m_h\pmod{F_i}$ ($0\le j<h\le F_i-1$), and if for a positive integer $M$, $M\equiv j$ and $M\ne m_j$ ($0\le j\le F_i-1$), then $M>m_j$. Then, for the case $p=1$ we shall construct the second set of the least complete residue system ${\rm Ap} (F_i;1)$. That is, $m_j^{(1)}\not\equiv m_h^{(1)}\pmod{F_i}$ ($0\le j<h\le F_i-1$), $m_j^{(1)}\equiv m_j\pmod{F_i}$ ($0\le j\le F_i-1$), and there does not exist an integer $M$ such that $m_j^{(1)}>M>m_j$ and $M\equiv j\pmod{F_i}$. 
Similarly, for $p=2$, we shall construct the third set of the least complete residue system ${\rm Ap} (F_i;2)$. That is, $m_j^{(2)}\not\equiv m_h^{(2)}\pmod{F_i}$ ($0\le j<h\le F_i-1$), $m_j^{(2)}\equiv m_j^{(1)}\pmod{F_i}$ ($0\le j\le F_i-1$), and there does not exist an integer $M$ such that $m_j^{(2)}>M>m_j^{(1)}$ and $M\equiv j\pmod{F_i}$.

By using a similar frame as in \cite{MAR}, we firstly show an analogous result about Lucas triple  $(L_i,L_{i+2},L_{i+k})$ when $p=0$. As a preparation we shall show the result when $p=0$, with a sketch of the proof.   
The results about Fibonacci numbers can be applied to get those about Lucas numbers.  When $p=0$, by setting integers $r$ and $\ell$ as $L_i-1=r F_k+\ell$ with $r\ge 0$ and $0\le\ell\le F_k-1$, and by using the identity $L_n=L_m F_{n-m+1}+L_{m-1}F_{n-m}$, we can get an analogous identity of the Fibonacci one in \cite[Theorem 1]{MAR}. 

\begin{theorem}
For integers $i,k\ge 3$ and $r=\fl{(L_i-1)/F_k}$, we have 
\begin{align*}
&g_0(L_i,L_{i+2},L_{i+k})\\
&=
 \left\{
    \begin{alignedat}{2}
        & (L_i-1)L_{i+2}-L_i(r F_{k-2}+1) &\quad&    \text{if $r=0$, or $r\ge 1$ and}  \\
&&&\text{$(L_i-r F_k)L_{i+2}>F_{k-2}L_i$},\\
        & (r F_k-1)L_{i+2}-L_i\bigl((r-1)F_{k-2}+1\bigr) &     &    \text{otherwise}\,.
    \end{alignedat}
    \right. 
\end{align*} 
\label{th:mar-lucas}
\end{theorem} 
\begin{proof}  
Consider the linear representation  
\begin{align*}
t_{x,y}:&=x L_{i+2}+y L_{i+k}\\
&=(x+y F_k)L_{i+2}-y F_{k-2}L_i\quad(x,y\ge 0)\,. 
\end{align*}
Then, by $\gcd(L_i,L_{i+2})=1$, we can prove that the above table represents the least complete residue system $\{0,1,\dots,L_i-1\}\pmod{L_i}$.   

\begin{table}[htbp]
  \centering
%\scalebox{0.7}{
\begin{tabular}{cccc}
\cline{1-2}\cline{3-4}
\multicolumn{1}{|c}{$t_{0,0}$}&$\cdots$&$\cdots$&\multicolumn{1}{c|}{$t_{F_k-1,0}$}\\
\multicolumn{1}{|c}{$t_{0,1}$}&$\cdots$&$\cdots$&\multicolumn{1}{c|}{$t_{F_k-1,1}$}\\
\multicolumn{1}{|c}{$\vdots$}&&&\multicolumn{1}{c|}{$\vdots$}\\
\multicolumn{1}{|c}{$t_{0,r-1}$}&$\cdots$&$\cdots$&\multicolumn{1}{c|}{$t_{F_k-1,r-1}$}\\
\cline{4-4}
\multicolumn{1}{|c}{$t_{0,r}$}&$\cdots$&\multicolumn{1}{c|}{$t_{\ell,r}$}&\\
\cline{1-2}\cline{2-3}
\end{tabular}
%} 
  %\caption{The complete residue systems for Lucas numbers}
  \label{tb:g0system}
\end{table}

That is, we can prove that none of two elements among this set is not congruent modulo $F_i$, and if there exists an element congruent to any of the elements among this set, then such an element is bigger and not in this set. 

When $r=0$, the largest element among all the $t_{x,y}$'s in this table is $t_{\ell,0}$. When $r\ge 1$, the largest element is either $t_{F_k-1,r-1}$ or $t_{\ell,r}$. Since $t_{F_k-1,r-1}<t_{\ell,r}$ is equivalent to $F_{k-2}L_i<(L_i-r F_k)L_{i+2}$, the result is followed by the identity (\ref{eq:bs}). 
The first case is given by $t_{\ell,r}-L_i$, and the second is given by $t_{F_k-1,r-1}-L_i$. 
\end{proof}

\section{Main results when $p=1$}  

Now, let us begin to consider the case $p\ge 1$. We shall obtain the $p$-Frobenius number using Lemma \ref{cor-mp} (\ref{mp-g}). For this it is necessary to know the structure of the elements of the $p$-Ap\'ery set, and the structure of the elements of the $p$-Ap\'ery set depends on the structure of the elements of the ($p-1$)-Ap\'ery set. 
 Therefore, in the case of $p=1$, the structure of the elements of the $1$-Ap\'ery set set is analyzed from the structure of the elements of the $0$-Ap\'ery set, which is the original Ap\'ery set, thereby obtaining the $1$-Frobenius number.  
When $p=1$, we have the following.  

\begin{theorem}  
For $i\ge 3$, we have 
\begin{align}  
g_1(F_i,F_{i+2},F_{i+k})&=(2 F_i-1)F_{i+2}-F_i\quad(k\ge i+2)\,, 
\label{eq:d1+2}\\
g_1(F_i,F_{i+2},F_{2 i+1})&=(F_{i-2}-1)F_{i+2}+F_{2 i+1}-F_i\,, 
\label{eq:d1+1}\\
g_1(F_i,F_{i+2},F_{2 i})&=(F_{i}-1)F_{i+2}+F_{2 i}-F_i\,. 
\label{eq:d1+0}
\end{align} 
When $r=\fl{(F_i-1)/F_k}\ge 1$, that is, $k\le i-1$, we have 
\begin{align}
&g_1(F_i,F_{i+2},F_{i+k})\notag\\
&=\begin{cases}
(F_i-r F_k-1)F_{i+2}+(r+1)F_{i+k}-F_i&\text{if $(F_i-r F_k)F_{i+2}\ge F_{k-2}F_i$},\\
(F_k-1)F_{i+2}+r F_{i+k}-F_i&\text{if $(F_i-r F_k)F_{i+2}<F_{k-2}F_i$}. 
\end{cases}
\label{eq:d1-123}  
\end{align}
\label{th:d1} 
\end{theorem}  

\noindent 
{\it Remark.} 
When $r\ge1$ and $k=i-1,i-2,i-3,i-4,i-5$, we have more explicit formulas.  
\begin{align*}
g_1(F_i,F_{i+2},F_{2 i-1})&=(F_{i-2}-1)F_{i+2}+2 F_{2 i-1}-F_i\quad(i\ge 4)\,, 
%\label{eq:d1-1}
\\
g_1(F_i,F_{i+2},F_{2 i-2})&=(F_{i-3}-1)F_{i+2}+3 F_{2 i-2}-F_i\quad(i\ge 5)\,, 
%\label{eq:d1-2}
\\
g_1(F_i,F_{i+2},F_{2 i-3})&=\begin{cases}
(F_{i-6}-1)F_{i+2}+5 F_{2 i-3}-F_i&\text{$(i\ge 7)$}\\
F_{i+2}+4 F_{2 i-3}-F_i(=149)&\text{$(i=6)$}\,,
\end{cases}
%\label{eq:d1-3}
\\
g_1(F_i,F_{i+2},F_{2 i-4})&=(F_{i-5}+F_{i-7}-1)F_{i+2}+7 F_{2 i-4}-F_i\,,  
%\label{eq:d1-4}
\\
g_1(F_i,F_{i+2},F_{2 i-5})&=\begin{cases}
(F_{i-5}-1)F_{i+2}+11 F_{2 i-5}-F_i&\text{$(i\ge 10)$}\\
12 F_{2 i-5}-F_i&\text{$(i=9)$}\\
11 F_{2 i-5}-F_i&\text{$(i=8)$}
\,.  
\end{cases}
%\label{eq:d1-5}
\end{align*}

\begin{proof} 
Put the linear representation  
\begin{align*}
t_{x,y}:&=x F_{i+2}+y F_{i+k}\\
&=(x+y F_k)F_{i+2}-y F_{k-2}F_i\quad(x,y\ge 0)\,. 
\end{align*}
For given $F_i$ and $F_k$, integers $r$ and $\ell$ are determined uniquely as $F_i-1=r F_k+\ell$ with $0\le\ell\le F_k-1$. 
 
\begin{table}[htbp]
  \centering
\scalebox{0.7}{
  \begin{tabular}{cccccccccccc}
\cline{1-2}\cline{3-4}\cline{5-6}\cline{7-8}\cline{9-10}\cline{11-12}
\multicolumn{1}{|c}{\shadowbox{$t_{0,0}$}}&\shadowbox{$t_{1,0}$}&$\cdots$&&$\cdots$&\multicolumn{1}{c|}{\shadowbox{$t_{F_k-1,0}$}}&$t_{F_k,0}$&$t_{F_k+1,0}$&$\cdots$&&$\cdots$&\multicolumn{1}{c|}{$t_{2 F_k-1,0}$}\\
\multicolumn{1}{|c}{$t_{0,1}$}&$t_{1,1}$&$\cdots$&&$\cdots$&\multicolumn{1}{c|}{$t_{F_k-1,1}$}&$t_{F_k,1}$&$t_{F_k+1,1}$&$\cdots$&&$\cdots$&\multicolumn{1}{c|}{$t_{2 F_k-1,1}$}\\
\multicolumn{1}{|c}{$t_{0,2}$}&$t_{1,2}$&$\cdots$&&$\cdots$&\multicolumn{1}{c|}{$t_{F_k-1,2}$}&$\vdots$&$\vdots$&&&&\multicolumn{1}{c|}{$\vdots$}\\
\multicolumn{1}{|c}{$\vdots$}&$\vdots$&&&&\multicolumn{1}{c|}{$\vdots$}&$t_{F_k,r-2}$&$t_{F_k+1,r-2}$&$\cdots$&&$\cdots$&\multicolumn{1}{c|}{\Ovalbox{$t_{2 F_k-1,r-2}$}}\\
\cline{10-11}\cline{11-12}
\multicolumn{1}{|c}{$t_{0,r-1}$}&$t_{1,r-1}$&$\cdots$&&$\cdots$&\multicolumn{1}{c|}{$t_{F_k-1,r-1}$}&$t_{F_k,r-1}$&$\cdots$&\multicolumn{1}{c|}{\Ovalbox{$t_{F_k+\ell,r-1}$}}&&&\\
\cline{4-5}\cline{5-6}\cline{7-8}\cline{8-9}
\multicolumn{1}{|c}{$t_{0,r}$}&$\cdots$&\multicolumn{1}{c|}{$t_{\ell,r}$}&$t_{\ell+1,r}$&$\cdots$&\multicolumn{1}{c|}{\Ovalbox{$t_{F_k-1,r}$}}&&&&&&\\
\cline{1-2}\cline{3-4}\cline{5-6}
\multicolumn{1}{|c}{$t_{0,r+1}$}&$\cdots$&\multicolumn{1}{c|}{\Ovalbox{$t_{\ell,r+1}$}}&&&&&&&&&\\
\cline{1-2}\cline{2-3}
  \end{tabular}
} 
  \caption{${\rm Ap}(F_i;0)$ and ${\rm Ap}(F_i;1)$ for $r\ge 1$}
  \label{tb:g1system}
\end{table}

The second set ${\rm Ap}(F_i;1)$ can be yielded from the first set ${\rm Ap}(F_i;0)$ as follows. Assume that $r\ge 1$. Only the first line $\{t_{0,0},t_{1,0},\dots,t_{F_k-1,0}\}$ moves to fill the last gap in the ($r+1$)-st line, and the rest continue to the next ($r+2$)-nd line.  
Everything else from the second line shifts up by $1$ and moves to the next right block (When $r=1$, the new right block consists of only one line $t_{F_k,0},\cdots,t_{F_k+\ell,0}$, but this does not affect the final result). 
\begin{align*}  
t_{0,1}&\equiv t_{F_k,0},&t_{1,1}&\equiv t_{F_k+1,0},&\quad\dots,\quad t_{F_k-1,1}&\equiv t_{2 F_k-1,0},\\ 
t_{0,2}&\equiv t_{F_k,1},&t_{1,2}&\equiv t_{F_k+1,1},&\quad\dots,\quad t_{F_k-1,2}&\equiv t_{2 F_k-1,1},\\ 
\dots\\
t_{0,r-1}&\equiv t_{F_k,r-2},&t_{1,r-1}&\equiv t_{F_k+1,r-2},&\quad\dots,\quad t_{F_k-1,r-1}&\equiv t_{2 F_k-1,r-2},\\ 
t_{0,r}&\equiv t_{F_k,r-1},&\!\!\!\!\quad\dots,&\!\!\!\!&t_{\ell,r}\equiv t_{F_k+\ell,r-1},&&
\end{align*} 
$$
\qquad\qquad\qquad\qquad t_{0,0}\equiv t_{\ell+1,r},\quad\dots,\quad t_{F_k-\ell-2,0}\equiv t_{F_k-1,r},
$$
$$
\!\!\!\!\!\!\!\!\!\!\!\! t_{F_k-\ell-1,0}\equiv t_{0,r+1},\quad\dots,\quad t_{F_k-1,0}\equiv t_{\ell,r+1}.
$$  
The first group is summarized as 
$$
t_{x,y}\equiv t_{F_k+x,y-1}\pmod{F_i} 
$$ 
for $0\le x\le F_k-1$ and $1\le y\le r-1$ or $0\le x\le\ell$ and $y=r$. This congruence is valid because 
\begin{align*}
&t_{x,y}=(x+y F_k)F_{i+2}-y F_{k-2}F_i\\
&\equiv(F_k+x+(y-1)F_k)F_{i+2}-(y-1)F_{k-2}F_i=t_{F_k+x,y-1}\pmod{F_i}\,.  
\end{align*}
The second group is valid because for $0\le x\le F_k-\ell-2$,  
\begin{align*}
t_{x,0}&=x F_{i+2}\\
&\equiv(\ell+1+x+r F_k)F_{i+2}-r F_{k-2}F_i=t_{\ell+1+x,r}\pmod{F_i}\,.  
\end{align*}
The third group is valid because for $0\le x\le \ell$,  
\begin{align*}
t_{F_k-\ell-1+x,0}&=(F_k-\ell-1+x)F_{i+2}\\
&\equiv(x+(r+1)F_k)F_{i+2}-(r+1)F_{k-2}F_i=t_{x,r+1}\pmod{F_i}\,.  
\end{align*}

\begin{table}[htbp]
  \centering
\scalebox{0.7}{
\begin{tabular}{ccccccc} 
\cline{1-2}\cline{3-4}\cline{5-6}\cline{6-7}
\multicolumn{1}{|c}{$t_{0,0}$}&$\cdots$&$\cdots$&\multicolumn{1}{c|}{$t_{\ell,0}$}&$t_{\ell+1,0}$&$\cdots$&\multicolumn{1}{c|}{\Ovalbox{$t_{F_k-1,0}$}}\\
\cline{1-2}\cline{3-4}\cline{5-6}\cline{6-7}
\multicolumn{1}{|c}{$t_{0,1}$}&$\cdots$&\multicolumn{1}{c|}{\Ovalbox{$t_{2\ell+1-F_k,1}$}}&&&&\\
\cline{1-2}\cline{2-3}
\end{tabular}
} 
  \caption{${\rm Ap}(F_i;0)$ and ${\rm Ap}(F_i;1)$ for $r=0$ and $2\ell+1\ge F_k$}
  \label{tb:g1system01}
\end{table}

Assume that $r=0$. The first set ${\rm Ap}(F_i;0)$ consists of only the first line. 
If $2\ell+1\ge F_k$, then  the second set ${\rm Ap}(F_i;1)$ can be yielded by moving to fill the last gap in the line, the rest continuing to the next line. 
\begin{align*}
t_{0,0}&\equiv t_{\ell+1,0},\quad &&\dots,&\quad t_{F_k-\ell-2,0}&\equiv t_{F_k-1,0}\,,\\
t_{F_k-\ell-1,0}&\equiv t_{0,1},\quad &&\dots,&\quad t_{\ell,0}&\equiv t_{2\ell+1-F_k,1}\pmod{F_i}\,. 
\end{align*} 
They are valid because for $0\le j\le F_k-\ell-2$, 
\begin{align*} 
t_{j,0}&=j F_{i+2}\equiv(F_i+j)F_{i+2}\\
&=(\ell+1-j)F_{i+2}=t_{\ell+1-j,0}\pmod{F_i}\,,
\end{align*}
and for $0\le j\le 2\ell+1-F_k$, 
\begin{align*} 
t_{F_k-\ell-1+j,0}&=(F_k-\ell-1+j)F_{i+2}\\
&\equiv(j+F_k)F_{i+2}-F_{k-2}F_i=t_{j,1}\pmod{F_i}\,. 
\end{align*}

\begin{table}[htbp]
  \centering
\scalebox{0.7}{
\begin{tabular}{cccccccc} 
\cline{1-2}\cline{3-4}\cline{5-6}\cline{7-8}
\multicolumn{1}{|c}{$t_{0,0}$}&$\cdots$&$\cdots$&\multicolumn{1}{c|}{$t_{\ell,0}$}&$t_{\ell+1,0}$&$\cdots$&$\cdots$&\multicolumn{1}{c|}{\Ovalbox{$t_{2\ell+1,0}$}}\\
\cline{1-2}\cline{3-4}\cline{5-6}\cline{7-8}
\end{tabular}
} 
  \caption{${\rm Ap}(F_i;0)$ and ${\rm Ap}(F_i;1)$ for $r=0$ and $2\ell+1\le F_k-1$}
  \label{tb:g1system00}
\end{table}

If $2\ell+1\le F_k-1$, then  the second set ${\rm Ap}(F_i;1)$ can be yielded by moving to fill the last gap in the line only. 
$$
t_{0,0}\equiv t_{\ell+1,0},\quad\dots,\quad t_{\ell,0}\equiv t_{2\ell+1,0}\pmod{F_i}\,.
$$ 
They are valid because for $0\le j\le\ell$,  
\begin{align*}  
t_{j,0}&=j F_{i+2}\equiv(F_i+j)F_{i+2}\\
&=(\ell+j+1)F_{i+2}=t_{\ell+j+1,0}\pmod{F_i}\,. 
\end{align*}

Next, we shall decide the maximal element in the second set ${\rm Ap}(F_i;1)$ (and also in the first set ${\rm Ap}(F_i;0)$).  

\noindent 
{\bf Case 1(1)} Assume that $r=0$ and $2\ell+1\le F_k-1$. The second condition is equivalent to $2 F_i\le F_k$, which is equivalent to $i\le k-2$.  
The largest element in the second set ${\rm Ap}(F_i;1)$, which is congruent to $\{0,1,\dots,F_i-1\}\pmod{F_i}$, is given by $t_{2\ell+1,0}=(2 F_i-1)F_{i+2}$. 

\noindent 
{\bf Case 1(2)} Assume that $r=0$ and $2\ell+1\ge F_k$. The second condition is equivalent to $2 F_i-1\ge F_k$, which is equivalent to $i\ge k-1\ge 3$. In this case there are two possibilities for the largest element in the second set ${\rm Ap}(F_i;1)$: $t_{F_k-1,0}=(F_k-1)F_{i+2}$ or $t_{2\ell+1-F_k,1}=(2 F_i-1)F_{i+2}-F_{k-2}F_i$. However, because of $i\ge k-1\ge 3$, always $t_{F_k-1,0}<t_{2\ell+1-F_k,1}$.  

\noindent 
{\bf Case 2} Assume that $r\ge 1$. This condition is equivalent to $F_i-1\ge F_k$, which is equivalent to $i\ge k+1$. In this case there are four possibilities for the largest element in the second set ${\rm Ap}(F_i;1)$: 
\begin{align*}  
t_{2 F_k-1,r-2}&=(r F_k-1)F_{i+2}-(r-2)F_{k-2}F_i,\\  
t_{F_k+\ell,r-1}&=(F_i-1)F_{i+2}-(r-1)F_{k-2}F_i,\\
t_{F_k-1,r}&=\bigl((r+1)F_k-1\bigr)F_{i+2}-r F_{k-2}F_i,\\
t_{\ell,r+1}&=(F_i+F_k-1)F_{i+2}-(r+1)F_{k-2}F_i. 
\end{align*}
However, it is clear that $t_{2 F_k-1,r-2}<t_{F_k-1,r}$. Because of $i\ge k+1$, $t_{F_k+\ell,r-1}<t_{F_k-1,r}$. Thus, the only necessity is to compare $t_{F_k-1,r}$ and $t_{\ell,r+1}$, and $t_{F_k-1,r}>t_{\ell,r+1}$ is equivalent to $(F_i-r F_k)F_{i+2}>F_{k-2}F_i$. 

Finally, rewriting the forms in terms of $F_{i+2}$ and $F_{i+k}$ and applying Lemma \ref{cor-mp} (\ref{mp-g}), we get the result. Namely, the formula (\ref{eq:d1+2}) comes from Case 1(1). The formulas (\ref{eq:d1+1}) and (\ref{eq:d1+0}) come from Case 1(2) when $k=i+1$ and $k=i$, respectively. The general formula (\ref{eq:d1-123}) comes from Case 2.   
\end{proof}

\section{The case $p=2$}

When $p=2$, we have the following.  

\begin{theorem}  
For $i\ge 3$, we have 
\begin{align}  
g_2(F_i,F_{i+2},F_{i+k})&=(3 F_i-1)F_{i+2}-F_i\quad(k\ge i+3)\,, 
\label{eq:d2+3}\\
g_2(F_i,F_{i+2},F_{2 i+2})&=\begin{cases}
(F_{i-2}-1)F_{i+2}+F_{2 i+2}-F_i&\text{$(i$ is odd$)$}\\
(F_{i+2}-1)F_{i+2}-F_i&\text{$(i$ is even$)$}\,,
\end{cases}
\label{eq:d2+2}\\
g_2(F_i,F_{i+2},F_{2 i+1})&=(F_{i}-1)F_{i+2}+F_{2 i+1}-F_i\,, 
\label{eq:d2+1}\\
g_2(F_i,F_{i+2},F_{2 i})&=(2 F_{i}-1)F_{i+2}-F_i\,, 
\label{eq:d2+0}\\
g_2(F_i,F_{i+2},F_{2 i-1})&=\begin{cases}
(F_{i-4}-1)F_{i+2}+3 F_{2 i-1}-F_i&\text{$(i\ge 5)$}\\
F_{i+2}+2 F_{2 i-1}-F_i(=31)&\text{$(i=4)$}\,. 
\end{cases}
\label{eq:d2-1}
\end{align} 
When $r=\fl{(F_i-1)/F_k}\ge 2$, that is, $k\le i-2$, we have 
\begin{align}
&g_2(F_i,F_{i+2},F_{i+k})\notag\\
&=\begin{cases}
(F_i-r F_k-1)F_{i+2}+(r+2)F_{i+k}-F_i&\text{if $(F_i-r F_k)F_{i+2}\ge F_{k-2}F_i$};\\
(F_k-1)F_{i+2}+(r+1)F_{i+k}-F_i&\text{if $(F_i-r F_k)F_{i+2}<F_{k-2}F_i$}\,. 
\end{cases}
\label{eq:d2-234}
\end{align} 
\label{th:d2} 
\end{theorem}  

\noindent 
{\it Remark.}  
When $k=i-2$ and $k=i-3$, we can write this more explicitly as 
\begin{align}
g_2(F_i,F_{i+2},F_{2 i-2})&=(F_{i-3}-1)F_{i+2}+4 F_{2 i-2}-F_i\quad(i\ge 5)\,,
\label{eq:d2-2}\\
g_2(F_i,F_{i+2},F_{2 i-3})&=\begin{cases}
(F_{i-6}-1)F_{i+2}+6 F_{2 i-3}-F_i&\text{$(i\ge 7)$}\\
F_{i+2}+5 F_{2 i-3}-F_i(=183)&\text{$(i=6)$}\,, 
\end{cases} 
\label{eq:d2-3}
\end{align} 
respectively. The formulas (\ref{eq:d2-2}) and (\ref{eq:d2-3}) hold when $r=2$ and $r=4$, respectively.

\begin{proof}
When $p=2$, the third least complete residue system ${\rm Ap}(F_i;2)$ is determined from the second least complete residue system ${\rm Ap}(F_i;1)$. When $r\ge 2$, some elements go to the third block.  
 
\begin{table}[htbp]
  \centering
\scalebox{1.0}{
\begin{tabular}{ccc}
\multicolumn{1}{|c}{1st block}&\multicolumn{1}{|c|}{2nd block}&\multicolumn{1}{c|}{3rd block}\\
\cline{1-2}\cline{2-3}
\multicolumn{1}{|c}{$\underbrace{\phantom{the 1st block}}_{F_k}$}&\multicolumn{1}{|c|}{$\underbrace{\phantom{the 1st block}}_{F_k}$}&\multicolumn{1}{c|}{$\underbrace{\phantom{the 1st block}}_{F_k}$}
\end{tabular}
} 
%  \caption{Block}
%  \label{tb:g2system-block}
\end{table}

\begin{table}[htbp]
  \centering
\scalebox{0.7}{
\begin{tabular}{ccccccc}
&&&&&&\multicolumn{1}{c|}{}\\ 
\cline{1-2}\cline{3-4}\cline{5-6}
\multicolumn{1}{|c}{}&\multicolumn{1}{c|}{\phantom{$t_{3\ell+2,0}$}}&&\multicolumn{1}{c|}{\phantom{$t_{3\ell+2,0}$}}&&\multicolumn{1}{c|}{$t_{3\ell+2,0}$}&\multicolumn{1}{c|}{}\\ 
\cline{1-2}\cline{3-4}\cline{5-6}
&&&&&&\multicolumn{1}{c|}{}
\end{tabular}
} 
  \caption{${\rm Ap}(F_i;p)$ ($p=0,1,2$) for $r=0$ and $F_k\ge 3\ell+3$}
  \label{tb:g2system200}
\end{table}

\noindent 
{\bf Case 1(1)} 
Let $r=0$ and $F_k\ge 3\ell+3=3 F_i$. Since for $\ell+1\le j\le 2\ell+1$, 
\begin{align*}
t_{j,0}&=j F_{i+2}\equiv(F_i+j)F_{i+2}\\
&=(\ell+j+1)F_{i+2}=t_{\ell+j+1,0}\pmod{F_i}\,, 
\end{align*} 
the third set ${\rm Ap}(F_i;2)$ is given by 
$$
\{t_{2\ell+1,0},\dots,t_{3\ell+2,0}\}\pmod{F_i}\,. 
$$ 
As the maximal element is $t_{3\ell+2,0}$, by (\ref{mp-g}), we have 
$$
g_2(F_i,F_{i+2},F_{i+k})=(3 F_i-1)F_{i+2}-F_i\,. 
$$ 

\begin{table}[htbp]
  \centering
\scalebox{0.7}{
\begin{tabular}{ccccccc}
&&&&&&\multicolumn{1}{c|}{}\\ 
\cline{1-2}\cline{3-4}\cline{5-6}\cline{6-7}
\multicolumn{1}{|c}{}&&\multicolumn{1}{c|}{\phantom{$t_{F_k-1,0}$}}&&&\multicolumn{1}{c|}{\phantom{$t_{F_k-1,0}$}}&\multicolumn{1}{c|}{$t_{F_k-1,0}$}\\ 
\cline{1-2}\cline{3-4}\cline{5-6}\cline{6-7}
\multicolumn{1}{|c}{}&\multicolumn{1}{c|}{$t_{3\ell+2-F_k,1}$}&&&&&\multicolumn{1}{c|}{}\\ 
\cline{1-2}
&&&&&&\multicolumn{1}{c|}{}
\end{tabular}
} 
  \caption{${\rm Ap}(F_i;p)$ ($p=0,1,2$) for $r=0$ and $2\ell+2\le F_k\le 3\ell+2$}
  \label{tb:g2system201}
\end{table}

\noindent 
{\bf Case 1(2)} 
Let $r=0$ and $2 F_i=2\ell+2\le F_k\le 3\ell+2=3 F_i-1$. 
Since $t_{j,0}\equiv t_{\ell+j+1,0}\pmod{F_i}$ ($\ell+1\le j\le F_k-\ell-2$), and for $0\le j\le 3\ell+2-F_k$, 
\begin{align*}   
t_{F_k-\ell-1+j,0}&=(F_k-\ell-1+j)F_{i+2}=(F_k-F_i+j)F_{i+2}\\
&\equiv(F_k+j)F_{i+2}\equiv(j+F_k)F_{i+2}+F_{k-2}F_i=t_{j,1}\pmod{F_i}\,,
\end{align*} 
the third set ${\rm Ap}(F_i;2)$ is    
$$
\{t_{2\ell+1,0},\dots,t_{F_k-1,0},t_{0,1},\dots,t_{3\ell+2-F_k,1}\}\pmod{F_i}\,.  
$$ 
The first elements $t_{2\ell+1,0},\dots,t_{F_k-1,0}$ are in the last of the first line, and the last elements $t_{0,1},,\dots,t_{3\ell+2-F_k,1}$ are in the first part of the second line. 
Hence, the maximal element is $t_{F_k-1,0}=(F_k-1)F_{i+2}$ or $t_{3\ell+2-F_k,1}=(3 F_i-1)F_{i+2}-F_{k-2}F_i$. 
Therefore, when $(3 F_i-F_k)F_{i+2}\ge F_{k-2}F_i$, $g_2(F_i,F_{i+2},F_{i+k})=(3 F_i-1)F_{i+2}-(F_{k-2}+1)F_i$. When $(3 F_i-F_k)F_{i+2}<F_{k-2}F_i$, $g_2(F_i,F_{i+2},F_{i+k})=(F_k-1)F_{i+2}-F_i$. 

\begin{table}[htbp]
  \centering
\scalebox{0.7}{
\begin{tabular}{cccccccccc}
&&&&&&\multicolumn{1}{c|}{}&&&\\  
\cline{1-2}\cline{3-4}\cline{5-6}\cline{7-8}\cline{9-10}
\multicolumn{1}{|c}{}&&&&\multicolumn{1}{c|}{}&&\multicolumn{1}{c|}{\phantom{$t_{\ell,1}$}}&&&\multicolumn{1}{c|}{$t_{2\ell+1,0}$}\\ 
\cline{1-2}\cline{3-4}\cline{5-6}\cline{7-8}\cline{9-10}
\multicolumn{1}{|c}{}&&\multicolumn{1}{c|}{\phantom{$t{2\ell+1,0}$}}&&\multicolumn{1}{c|}{$t_{\ell,1}$}&&\multicolumn{1}{c|}{}&&&\\   
\cline{1-2}\cline{3-4}\cline{4-5}
&&&&&&\multicolumn{1}{c|}{}&&&
\end{tabular}
} 
  \caption{${\rm Ap}(F_i;p)$ ($p=0,1,2$) for $r=0$ and $F_k\le 2\ell+1$}
  \label{tb:g2system202}
\end{table}

\noindent 
{\bf Case 1(3)} 
Let $r=0$ and $F_k\le 2\ell+1=2 F_i-1$. Since for $\ell+1\le j\le F_k-1$ 
\begin{align*}  
t_{j,0}&=j F_{i+2}\equiv(F_i+j)F_{i+2}\\
&\equiv(\ell+1+j)F_{i+2}-F_{k-2}F_i=t_{\ell-F_k+1+j,1}\pmod{F_i}\,, 
\end{align*} 
and for $0\le j\le 2\ell+1-F_k$
\begin{align*}  
t_{j,1}&=(j+F_k)F_{i-2}-F_{k-2}F_i\\
&\equiv(F_k+j)F_{i+2}=t_{F_k+j,0}\pmod{F_i}\,, 
\end{align*}
the third set ${\rm Ap}(F_i;2)$ is    
$$
\{t_{2\ell+2-F_k,1},\dots,t_{\ell,1},t_{F_k,0},\dots,t_{2\ell+1,0}\}\pmod{F_i}\,.
$$ 
The first elements $t_{2\ell+2-F_k,1},\dots,t_{\ell,1}$ are in the second line of the first block, and the last elements $t_{F_k,0},\dots,t_{2\ell+1,0}$ in the first line of the second block. 
So, the maximal element is $t_{\ell,1}=(F_i+F_k-1)F_{i+2}-F_{k-2}F_i$ or $t_{2\ell+1,0}=(2 F_i-1)F_{i+2}$. Since $F_k\le 2 F_i-1$, only when $k=i+1$, we have $g_2(F_i,F_{i+2},F_{i+k})=(F_i+F_k-1)F_{i+2}-(F_{k-2}+1)F_i$. When $k\le i$, we have $g_2(F_i,F_{i+2},F_{i+k})=(2 F_i-1)F_{i+2}-F_i$.

\begin{table}[htbp]
  \centering
\scalebox{0.7}{
\begin{tabular}{cccccccccccccc}
&&&&&&\multicolumn{1}{c|}{}&&&&&&&\multicolumn{1}{c|}{}\\
\cline{1-2}\cline{3-4}\cline{5-6}\cline{7-8}\cline{9-10}\cline{11-12}\cline{13-14}
\multicolumn{1}{|c}{}&&&&&&\multicolumn{1}{c|}{}&&\multicolumn{1}{c|}{\phantom{$t_{2\ell+1,2}$}}&&&&&\multicolumn{1}{c|}{$t_{2 F_k-1,0}$}\\
\cline{3-4}\cline{5-6}\cline{7-8}\cline{9-10}\cline{11-12}\cline{13-14}
\multicolumn{1}{|c}{}&\multicolumn{1}{c|}{}&&&&&\multicolumn{1}{c|}{}&&\multicolumn{1}{c|}{$t_{F_k+\ell,1}$}&&&&&\multicolumn{1}{c|}{}\\
\cline{1-2}\cline{3-4}\cline{5-6}\cline{7-8}\cline{8-9}
\multicolumn{1}{|c}{}&\multicolumn{1}{c|}{\phantom{$t_{F_k+\ell,1}$}}&&\multicolumn{1}{c|}{$t_{2\ell+1,2}$}&&&\multicolumn{1}{c|}{}&&&&&&&\multicolumn{1}{c|}{}\\
\cline{1-2}\cline{3-4}
&&&&&&\multicolumn{1}{c|}{}&&&&&&&\multicolumn{1}{c|}{}
\end{tabular}
} 
  \caption{${\rm Ap}(F_i;p)$ ($p=0,1,2$) for $r=1$ and $F_k\ge 2\ell+2$}
  \label{tb:g2system210}
\end{table}

\noindent 
{\bf Case 2(1)} 
Let $r=1$ and $2\ell+2\le F_k$, that is, $\frac{2}{3}F_i\le F_k\le F_i-1$. This case happens only when $i=4$ and $k=3$. 
Since for $\ell+1\le j\le F_k-1$ 
\begin{align*}
t_{j,1}&=(j+F_k)F_{i+2}-F_{k-2}F_i\\ 
&\equiv(F_k+j)F_{i+2}=t_{F_k+j,0}\pmod{F_i}\,,
\end{align*}
for $0\le j\le\ell$ 
\begin{align*}
t_{j,2}&=(j+2 F_k)F_{i+2}-2 F_{k-2}F_i\\ 
&\equiv(F_k+j+F_k)F_{i+2}-F_{k-2}F_i=t_{F_k+j,1}\pmod{F_i}\,,
\end{align*}
and for $0\le j\le\ell$ 
\begin{align*}
t_{F_k+j,0}&=(F_k+j)F_{i+2}\equiv(F_i+j+F_k)F_{i+2}\\ 
&\equiv(\ell+1+j+2 F_k)F_{i+2}-2 F_{k-2}F_i=t_{\ell+1+j,2}\pmod{F_i}\,,
\end{align*}
the third set ${\rm Ap}(F_i;2)$ is    
$$
\{t_{F_k+\ell+1,0},\dots,t_{2 F_k-1,0},t_{F_k,1},\dots,t_{F_k+\ell,1},t_{\ell+1,2},\dots,t_{2\ell+1,2}\}\pmod{F_i}\,.
$$ 
The first elements $t_{F_k+\ell+1,0},\dots,t_{2 F_k-1,0}$ are in the first line of the second block, the second elements $t_{F_k,1},\dots,t_{F_k+\ell,1}$ are in the second line of the second block, and the last elements $t_{\ell+1,2},\dots,t_{2\ell+1,2}$ are in the third line of the first block. 
So, the maximal element is one of $t_{2 F_k-1,0}=(2 F_k-1)F_{i+2}$, $t_{F_k+\ell,1}=(F_i+F_k-1)F_{i+2}-F_{k-2}F_i$ or $t_{2\ell+1,2}=(2 F_i-1)F_{i+2}-2 F_{k-2}F_i$. 
As $i=4$ and $k=3$, $t_{2\ell+1,2}=34$ is the largest. Hence, $g_2(F_i,F_{i+2},F_{i+k})=(2 F_i-1)F_{i+2}-(2 F_{k-2}+1)F_i$, that is, $g_2(F_4,F_6,F_7)=34-F_4=31$.  

\begin{table}[htbp]
  \centering
\scalebox{0.7}{
\begin{tabular}{cccccccccccccc}
&&&&&&\multicolumn{1}{c|}{}&&&&&&&\multicolumn{1}{c|}{}\\
\cline{1-2}\cline{3-4}\cline{5-6}\cline{7-8}\cline{9-10}\cline{11-12}\cline{13-14}
\multicolumn{1}{|c}{}&&&&&&\multicolumn{1}{c|}{}&&&&&\multicolumn{1}{c|}{}&&\multicolumn{1}{c|}{$t_{2 F_k-1,0}$}\\
\cline{6-7}\cline{7-8}\cline{9-10}\cline{11-12}\cline{13-14}
\multicolumn{1}{|c}{}&&&&\multicolumn{1}{c|}{}&&\multicolumn{1}{c|}{}&&&&&\multicolumn{1}{c|}{\phantom{$t_{F}$}$t_{F_k+\ell,1}$}&&\multicolumn{1}{c|}{}\\
\cline{1-2}\cline{3-4}\cline{5-6}\cline{7-8}\cline{9-10}\cline{11-12}
\multicolumn{1}{|c}{}&&&&\multicolumn{1}{c|}{}&&\multicolumn{1}{c|}{$t_{F_k-1,2}$}&&&&&&&\multicolumn{1}{c|}{}\\
\cline{1-2}\cline{3-4}\cline{5-6}\cline{6-7}
\multicolumn{1}{|c}{}&\multicolumn{1}{c|}{$t_{2\ell+1-F_k,3}$}&&&&&\multicolumn{1}{c|}{}&&&&&&&\multicolumn{1}{c|}{}\\
\cline{1-2}
&&&&&&\multicolumn{1}{c|}{}&&&&&&&\multicolumn{1}{c|}{}
\end{tabular}
} 
  \caption{${\rm Ap}(F_i;p)$ ($p=0,1,2$) for $r=1$ and $F_k\le 2\ell+1$}
  \label{tb:g2system211}
\end{table}

\noindent 
{\bf Case 2(2)} 
Let $r=1$ and $2\ell+1\ge F_k$, that is, $(F_i-1)/2<F_k\le(2F_i-1)/3$. This relation holds only when $k=i-1\ge 4$.  
Since $t_{j,1}\equiv t_{F_k+j,0}\pmod{F_i}$ ($\ell+1\le j\le F_k-1$), $t_{j,2}\equiv t_{F_k+j,1}\pmod{F_i}$ ($0\le j\le\ell$), $t_{F_k+j,0}\equiv t_{\ell+1+j,2}\pmod{F_i}$ ($0\le j\le F_k-\ell-2$), and for $0\le j\le 2\ell+1-F_k$ 
\begin{align*}
t_{2 F_k-\ell-1+j,0}&=(2 F_k-\ell-1+j)F_{i+2}\equiv(3 F_k-F_i+j)F_{i+2}\\ 
&\equiv(j+3 F_k)F_{i+2}-3 F_{k-2}F_i=t_{j,3}\pmod{F_i}\,,
\end{align*}
the third set ${\rm Ap}(F_i;2)$ is  
\begin{multline*}
\{t_{F_k+\ell+1,0},\dots,t_{2 F_k-1,0},t_{F_k,1},\dots,t_{F_k+\ell,1},\\
t_{\ell+1,2},\dots,t_{F_k-1,2},t_{0,3},\dots,t_{2\ell+1-F_k,3}\}\pmod{F_i}\,.
\end{multline*}
So, the maximal element is one of $t_{2 F_k-1,0}=(2 F_k-1)F_{i+2}$, $t_{F_k+\ell,1}=(F_i+F_k-1)F_{i+2}-F_{k-2}F_i$, $t_{F_k-1,2}=(3 F_k-1)F_{i+2}-2 F_{k-2}F_i$ or $t_{2\ell+1-F_k,3}=(2 F_i-1)F_{i+2}-3 F_{k-2}F_i$. 
As $k=i-1\ge 4$, $t_{2\ell+1-F_k,3}=(2 F_i-1)F_{i+2}-3 F_{i-3}F_i$ is the largest. Hence, $g_2(F_i,F_{i+2},F_{2 i-1})=(2 F_i-1)F_{i+2}-(3 F_{i-3}+1)F_i$.

\begin{table}[htbp]
  \centering
\scalebox{0.7}{
\begin{tabular}{ccccccccccccccc}
&&&&\multicolumn{1}{c|}{}&&&&&\multicolumn{1}{c|}{}&&&&&\multicolumn{1}{c|}{}\\ 
\cline{1-2}\cline{3-4}\cline{5-6}\cline{7-8}\cline{9-10}\cline{11-12}\cline{13-14}\cline{14-15}
\multicolumn{1}{|c}{}&&&&\multicolumn{1}{c|}{}&&&&&\multicolumn{1}{c|}{}&&&&&\multicolumn{1}{c|}{}\\ 
\multicolumn{1}{|c}{}&&&&\multicolumn{1}{c|}{}&&&&&\multicolumn{1}{c|}{}&&&&&\multicolumn{1}{c|}{$t_{3 F_k-1,r-3}$}\\ 
\cline{12-13}\cline{14-15}
\multicolumn{1}{|c}{}&&&&\multicolumn{1}{c|}{}&&&&&\multicolumn{1}{c|}{}&\multicolumn{1}{c|}{$t_{2 F_k+\ell,r-2}$}&&&&\multicolumn{1}{c|}{}\\ 
\cline{7-8}\cline{9-10}\cline{10-11}
\multicolumn{1}{|c}{}&&&&\multicolumn{1}{c|}{}&\multicolumn{1}{c|}{}&&&&\multicolumn{1}{c|}{$t_{2 F_k-1,r-1}$}&&&&&\multicolumn{1}{c|}{}\\ 
\cline{3-4}\cline{5-6}\cline{7-8}\cline{9-10}
\multicolumn{1}{|c}{}&\multicolumn{1}{c|}{}&&&\multicolumn{1}{c|}{}&\multicolumn{1}{c|}{\phantom{$t_F$}$t_{F_k+\ell,r}$}&&&&\multicolumn{1}{c|}{}&&&&&\multicolumn{1}{c|}{}\\ 
\cline{1-2}\cline{3-4}\cline{5-6}
\multicolumn{1}{|c}{}&\multicolumn{1}{c|}{}&&&\multicolumn{1}{c|}{$t_{F_k-1,r+1}$}&&&&&\multicolumn{1}{c|}{}&&&&&\multicolumn{1}{c|}{}\\ 
\cline{1-2}\cline{3-4}\cline{4-5}
\multicolumn{1}{|c}{}&\multicolumn{1}{c|}{\phantom{$t_F$}$t_{\ell,r+2}$}&&&\multicolumn{1}{c|}{}&&&&&\multicolumn{1}{c|}{}&&&&&\multicolumn{1}{c|}{}\\ 
\cline{1-2}
&&&&\multicolumn{1}{c|}{}&&&&&\multicolumn{1}{c|}{}&&&&&\multicolumn{1}{c|}{}
\end{tabular}
} 
  \caption{${\rm Ap}(F_i;p)$ ($p=0,1,2$) for $r\ge 2$}
  \label{tb:g2system300}
\end{table}

\noindent 
{\bf Case 3}  
Let $r\ge 2$. 
The part 
$$ 
t_{F_k,1},\dots,t_{2 F_k-1,1},\dots,t_{F_k,r-2},\dots,t_{2 F_k-1,r-2},t_{F_k,r-1},\dots,t_{F_k+\ell,r-1}
$$ 
in the second block among the second set ${\rm Ap}(F_i;1)$ corresponds to the part 
$$
t_{2 F_k,0},\dots,t_{3 F_k-1,0},\dots,t_{2 F_k,r-3},\dots,t_{3 F_k-1,r-3},t_{2 F_k,r-2},\dots,t_{2 F_k+\ell,r-2}
$$ 
in the third block among the third least set ${\rm Ap}(F_i;2)$\footnote{When $r=2$, only the last shorter line remains, and $t_{3 F_k-1,r-3}$ in table \ref{tb:g2system300} does not appear. However, this does not affect the result.} 
because 
\begin{align*}
t_{F_k+j,h}&=(F_k+j+h F_k)F_{i+2}-h F_{k-2}F_i\\
&\equiv\bigl(2 F_k+j+(h-1)F_k\bigr)F_{i+2}-(h-1)F_{k-2}F_i\\
&=t_{2 F_k+j,h-1}\pmod{F_i}
\end{align*}
for $0\le j\le F_k-1$ and $1\le h\le r-2$ or $0\le j\le\ell$ and $h=r-1$.  
The part $t_{\ell+1,r},\dots,t_{F_k-1,r}$ in the first block among the second set ${\rm Ap}(F_i;1)$ corresponds to the part $t_{F_k+\ell+1,r-1},\dots,t_{2 F_k-1,r-1}$ in the second block among the third set ${\rm Ap}(F_i;2)$ because for $\ell+1\le j\le F_k-1$ 
\begin{align*}
t_{j,r}&=(j+r F_k)F_{i+2}-r F_{k-2}F_i\\
&\equiv\bigl(F_k+j+(r-1)F_k\bigr)F_{i+2}-(r-1)F_{k-2}F_i\\
&=t_{F_k+j,r-1}\pmod{F_i}\,.  
\end{align*} 
The part $t_{0,r+1},\dots,t_{\ell,r+1}$ in the first block among the second set ${\rm Ap}(F_i;1)$ corresponds to the part $t_{F_k,r},\dots,t_{F_k+\ell,r}$ in the second block among the third set ${\rm Ap}(F_i;2)$ because for $0\le j\le\ell$ 
\begin{align*}
t_{j,r+1}&=\bigl(j+(r+1)F_k\bigr)F_{i+2}-(r+1)F_{k-2}F_i\\
&\equiv(F_k+j+r F_k)F_{i+2}-r F_{k-2}F_i\\
&=t_{F_k+j,r}\pmod{F_i}\,.  
\end{align*} 
The first line $t_{F_k,0},\dots,t_{2 F_k-\ell-2},t_{2 F_k-\ell-1},\dots,t_{2 F_k-1,0}$ in the second block among the second set ${\rm Ap}(F_i;1)$ corresponds to two parts $t_{\ell+1,r+1},\dots,t_{F_k-1,r+1}$ and $t_{0,r+2},\dots,t_{\ell,r+2}$ in the first block among the third set ${\rm Ap}(F_i;2)$ because for $0\le j\le F_k-\ell-2$
\begin{align*}
t_{F_k+j,0}&=(F_k+j)F_{i+2}\equiv(F_i+F_k+j)F_{i+2}-(r+1)F_{k-2}F_i\\
&=\bigl(\ell+1+j+(r+1)F_k\bigr)F_{i+2}-(r+1)F_{k-2}F_i\\
&=t_{\ell+1+j,r+1}\pmod{F_i}\,,   
\end{align*} 
and for $0\le j\le\ell$, 
\begin{align*}
t_{2 F_k-\ell-1+j,0}&=(2 F_k-\ell-1+j)F_{i+2}=(2 F_k-F_i+r F_k+j)F_{i+2}\\
&\equiv\bigl(j+(r+2)F_k\bigr)F_{i+2}-(r+2)F_{k-2}F_i\\
&=t_{j,r+2}\pmod{F_i}\,.   
\end{align*}   
Hence, the third set ${\rm Ap}(F_i;2)$ is given by 
\begin{multline*}
\{t_{2 F_k,0},\dots,t_{3 F_k-1,0},\dots,t_{2 F_k,r-3},\dots,t_{3 F_k-1,r-3},t_{2 F_k,r-2},\dots,t_{2 F_k+\ell,r-2},\\
t_{F_k+\ell+1,r-1},\dots,t_{2 F_k-1,r-1},t_{F_k,r},\dots,t_{F_k+\ell,r},\\
t_{\ell+1,r+1},\dots,t_{F_k-1,r+1},t_{0,r+2},\dots,t_{\ell,r+2}\}\pmod{F_i}\,.
\end{multline*}
There are six candidates for the maximal element: 
\begin{align*}
t_{3 F_k-1,r-3}&=(r F_k-1)F_{i+2}-(r-3)F_{k-2}F_i\,,\\
t_{2 F_k+\ell,r-2}&=(F_i-1)F_{i+2}-(r-2)F_{k-2}F_i\,,\\
t_{2 F_k-1,r-1}&=\bigl((r+1)F_k-1\bigr)F_{i+2}-(r-1)F_{k-2}F_i\,,\\
t_{F_k+\ell,r}&=(F_i+F_k-1)F_{i+2}-r F_{k-2}F_i\,,\\
t_{F_k-1,r+1}&=\bigl((r+2)F_k-1\bigr)F_{i+2}-(r+1)F_{k-2}F_i\,,\\
t_{\ell,r+2}&=(F_i+2 F_k-1)F_{i+2}-(r+2)F_{k-2}F_i\,. 
\end{align*} 
However, it is easy to see that the first four values are less than the last two. 
Hence, if $(F_i-r F_k)F_{i+2}\ge F_{k-2}F_i$, then $g_2(F_i,F_{i+2},F_{i+k})=t_{\ell,r+2}-F_i=(F_i+2 F_k-1)F_{i+2}-\bigl((r+2)F_{k-2}+1\bigr)F_i$. If $(F_i-r F_k)F_{i+2}<F_{k-2}F_i$, then $g_2(F_i,F_{i+2},F_{i+k})=t_{F_k-1,r+1}-F_i=\bigl((r+2)F_k-1\bigr)F_{i+2}-\bigl((r+1)+1\bigr)F_{k-2}F_i$. 

Finally we rewrite the form as the linear combination of $F_{i+2}$ and $F_{i+k}$ and apply Lemma \ref{cor-mp} (\ref{mp-g}). 
The formula (\ref{eq:d2+3}) comes from Case 1(1). The formula (\ref{eq:d2+2}) comes from Case 1(2). The formulas (\ref{eq:d2+1}) and (\ref{eq:d2+0}) come from Case 1(3). The formula (\ref{eq:d2-1}) comes from Case 2(1)(2). The general formula (\ref{eq:d2-234}) comes from Case 3. 
\end{proof}

\section{The case $p=3$} 

When $p=3$, we have the following.  

\begin{theorem}  
For $i\ge 3$, we have 
\begin{align}  
g_3(F_i,F_{i+2},F_{i+k})&=(4 F_i-1)F_{i+2}-F_i\quad(k\ge i+3)\,, 
\label{eq:d3+3}\\
g_3(F_i,F_{i+2},F_{2 i+2})&=(F_{i}-1)F_{i+2}+F_{2 i+2}-F_i\,, 
%\label{eq:d3+2}
\notag\\
g_3(F_i,F_{i+2},F_{2 i+1})&=(F_{i}+F_{i-2}-1)F_{i+2}+F_{2 i+1}-F_i\,, 
%\label{eq:d3+1}
\notag\\
g_3(F_i,F_{i+2},F_{2 i})&=(F_{i}-1)F_{i+2}+2 F_{2 i}-F_i\,, 
%\label{eq:d3+0}
\notag\\
g_3(F_i,F_{i+2},F_{2 i-1})&=(F_{i-2}-1)F_{i+2}+3 F_{2 i-1}-F_i\quad(i\ge 4)\,, 
%\label{eq:d3-1}
\notag\\
g_3(F_i,F_{i+2},F_{2 i-2})&=\begin{cases}
(F_{i-5}-1)F_{i+2}+5 F_{2 i-2}-F_i&\text{$(i\ge 6)$}\\
F_{i+2}+4 F_{2 i-2}-F_i(=92)&\text{$(i=5)$}\,. 
\end{cases}
\label{eq:d3-2}
%\\
%g_3(F_i,F_{i+2},F_{2 i-3})&=\begin{cases}
%(F_{i-6}-1)F_{i+2}+7 F_{2 i-3}-F_i&\text{$(i\ge 7)$}\\
%F_{i+2}+6 F_{2 i-3}-F_i(=217)&\text{$(i=6)$}\,. 
%\end{cases} 
%\label{eq:d3-3}
\end{align} 
When $r=\fl{(F_i-1)/F_k}\ge 3$, that is, $k\le i-3$, we have 
\begin{align} 
&g_3(F_i,F_{i+2},F_{i+k})\notag\\
&=\begin{cases}
(F_i-r F_k-1)F_{i+2}+(r+3)F_{i+k}-F_i&\text{if $(F_i-r F_k)F_{i+2}\ge F_{k-2}F_i$};\\
(F_k-1)F_{i+2}+(r+2)F_{i+k}-F_i&\text{if $(F_i-r F_k)F_{i+2}<F_{k-2}F_i$}\,. 
\end{cases}
\label{eq:d3-345}
\end{align}
\label{th:d3} 
\end{theorem}

\begin{proof}  
When $p=3$, the fourth least complete residue system ${\rm Ap}(F_i;3)$ is determined from the third least complete residue system ${\rm Ap}(F_i;2)$. When $r\ge 3$, some elements go to the fourth block. The proof of the cases $r=0,1,2$ is similar to that of Theorem \ref{th:d2} and needs more case-by-case discussions, and it is omitted.

\begin{table}[htbp]
  \centering
\scalebox{0.7}{
\begin{tabular}{cccccccccccccccc}
\multicolumn{1}{|c}{}&&&\multicolumn{1}{c|}{}&&&&\multicolumn{1}{c|}{}&&&&\multicolumn{1}{c|}{}&&&&\multicolumn{1}{c|}{}\\ 
\cline{1-2}\cline{3-4}\cline{5-6}\cline{7-8}\cline{9-10}\cline{11-12}\cline{13-14}\cline{15-16}
\multicolumn{1}{|c}{}&&&\multicolumn{1}{c|}{}&&&&\multicolumn{1}{c|}{}&&&&\multicolumn{1}{c|}{}&&&&\multicolumn{1}{c|}{}\\ 
\multicolumn{1}{|c}{}&&&\multicolumn{1}{c|}{}&&&&\multicolumn{1}{c|}{}&&$\ctext{3}$&&\multicolumn{1}{c|}{}&&$\ctext{4}$&&\multicolumn{1}{c|}{}\\ 
\cline{14-15}\cline{15-16}
\multicolumn{1}{|c}{}&&&\multicolumn{1}{c|}{}&&$\ctext{2}$&&\multicolumn{1}{c|}{}&&&&\multicolumn{1}{c|}{}&\multicolumn{1}{c|}{}&&&\multicolumn{1}{c|}{}\\ 
\cline{10-11}\cline{12-13}
\multicolumn{1}{|c}{}&$\ctext{1}$&&\multicolumn{1}{c|}{}&&&&\multicolumn{1}{c|}{}&\multicolumn{1}{c|}{}&&$\ctext{4}$&\multicolumn{1}{c|}{}&&&&\multicolumn{1}{c|}{}\\ 
\cline{6-6}\cline{7-8}\cline{9-10}\cline{11-12}
\multicolumn{1}{|c}{}&&&\multicolumn{1}{c|}{}&\multicolumn{1}{c|}{}&&$\ctext{3}$&\multicolumn{1}{c|}{}&\multicolumn{1}{c|}{$\ctext{4}$}&&&\multicolumn{1}{c|}{}&&&&\multicolumn{1}{c|}{}\\ 
\cline{2-3}\cline{4-5}\cline{6-7}\cline{8-9}
\multicolumn{1}{|c|}{}&&$\ctext{2}$&\multicolumn{1}{c|}{}&\multicolumn{1}{c|}{$\ctext{3}$}&&$\ctext{4}$&\multicolumn{1}{c|}{}&&&&\multicolumn{1}{c|}{}&&&&\multicolumn{1}{c|}{}\\ 
\cline{1-2}\cline{3-4}\cline{5-6}\cline{7-8}
\multicolumn{1}{|c|}{$\ctext{2}$}&&$\ctext{3}$&\multicolumn{1}{c|}{}&\multicolumn{1}{c|}{$\ctext{4}$}&&&\multicolumn{1}{c|}{}&&&&\multicolumn{1}{c|}{}&&&&\multicolumn{1}{c|}{}\\ 
\cline{1-2}\cline{3-4}\cline{4-5}
\multicolumn{1}{|c|}{$\ctext{3}$}&&$\ctext{4}$&\multicolumn{1}{c|}{}&&&&\multicolumn{1}{c|}{}&&&&\multicolumn{1}{c|}{}&&&&\multicolumn{1}{c|}{}\\ 
\cline{1-2}\cline{3-4}
\multicolumn{1}{|c|}{$\ctext{4}$}&&&\multicolumn{1}{c|}{}&&&&\multicolumn{1}{c|}{}&&&&\multicolumn{1}{c|}{}&&&&\multicolumn{1}{c|}{}\\ 
\cline{1-1}
\multicolumn{1}{|c}{}&&&\multicolumn{1}{c|}{}&&&&\multicolumn{1}{c|}{}&&&&\multicolumn{1}{c|}{}&&&&\multicolumn{1}{c|}{}
\end{tabular}
} 
  \caption{${\rm Ap}(F_i;p)$ ($p=0,1,2,3$) for $r\ge 3$}
  \label{tb:g3system400}
\end{table}

In the table, $\ctext{$n$}$ denotes the area of the $n$-th least set of the complete residue system ${\rm Ap} (F_i;n-1)$. Here, each $m_j^{(n-1)}$, satisfying $m_j^{(n-1)}\equiv j\pmod{F_i}$ ($0\le j\le F_i-1$), can be expressed in at least $n$ ways, but $m_j^{(n-1)}-F_i$ can be expressed in at most $n-1$ ways.  As illustrated in the proof of Theorem \ref{th:d2}, two areas (lines) of $\ctext{4}$ in the first block corresponds to the first line of $\ctext{3}$ in the third block, two areas (lines) of $\ctext{4}$ in the second block corresponds to two areas (lines) of $\ctext{3}$ in the first block, two areas (lines) of $\ctext{4}$ in the third block corresponds to two areas (lines) of $\ctext{3}$ in the second block, and the area of $\ctext{4}$ in the fourth block corresponds to the area of $\ctext{3}$ in the third block except the first line.   Eventually, the maximal element of the fourth set of the complete residue system is from the first block, that is, $t_{F_k-1,r+2}=\bigl((r+3)F_k-1\bigr)F_{i+2}-(r+2)F_{k-2}F_i$ or $t_{\ell,r+3}=(F_i+3 F_k-1)F_{i+2}-(r+3)F_{k-2}F_i$. Hence, if $(F_i-r F_k)F_{i+2}\ge F_{k-2}F_i$, then $g_3(F_i,F_{i+2},F_{i+k})=t_{\ell,r+3}-F_i=(F_i+3 F_k-1)F_{i+2}-\bigl((r+3)F_{k-2}+1\bigr)F_i$. If $(F_i-r F_k)F_{i+2}<F_{k-2}F_i$, then $g_3(F_i,F_{i+2},F_{i+k})=t_{F_k-1,r+2}-F_i=\bigl((r+3)F_k-1\bigr)F_{i+2}-\bigl((r+2)F_{k-2}+1\bigr)F_i$. 
Notice that $r\ge 3$ implies that $k\le i-3$.  
\end{proof}

\section{General $p$ case}  

Repeating the same process, when $r$ is big enough that $r\ge p$, that is, $k$ is comparatively smaller than $i$, as a generalization of (\ref{eq:d1-123}), (\ref{eq:d2-234}) and (\ref{eq:d3-345}), we can have an explicit formula.  

\begin{theorem}  
Let $i\ge 3$ and $p$ be a nonnegative integer.   
When $r=\fl{(F_i-1)/F_k}\ge p$ with $(r,p)\ne(0,0)$, we have 
\begin{align*} 
&g_p(F_i,F_{i+2},F_{i+k})\notag\\
&=\begin{cases}
(F_i-r F_k-1)F_{i+2}+(r+p)F_{i+k}-F_i&\text{if $(F_i-r F_k)F_{i+2}\ge F_{k-2}F_i$};\\
(F_k-1)F_{i+2}+(r+p-1)F_{i+k}-F_i&\text{if $(F_i-r F_k)F_{i+2}<F_{k-2}F_i$}\,. 
\end{cases}
%\label{eq:dp-ppp}
\end{align*}
\label{th:dp} 
\end{theorem} 

\noindent 
{\it Remark.}  
When $p=0$, Theorem \ref{th:dp} reduces to \cite[Theorem 1]{MAR} except $r=0$.

On the other hand, when $k$ is comparatively larger than $i$, as a generalization of (\ref{eq:d1+2}), (\ref{eq:d2+3}) and (\ref{eq:d3+3}), we can also have the following formula.   

\begin{Prop}  
For $i,k\ge 3$, we have 
$$ 
g_p(F_i,F_{i+2},F_{i+k})=g_p(F_i,F_{i+2})\quad(k\ge i+h)
$$ 
when $(p,h)=(3,4)$, $(4,4)$, $(5,5)$, $(6,5)$, $(7,5)$, $(8,5)$, $(9,6)$, $(10,6)$, $(11,6)$, $(12,6)$, $(13,6)$, $(14,6)$, $(15,7)$, $(16,7)$, $(17,7)$, $(18,7)$, $(19,7)$, $(20,7)$, $(21,7)$, $(22,7)$, $(23,7)$, $(24,8)$, $\dots$. 
\end{Prop}

The proof depends on the fact 
$$
(p+1)F_i F_{i+2}-F_i-F_{i+2}<F_{2 i+h}\quad(i\ge 3)\,.  
$$ 
Nevertheless, such $h$'s are not necessarily sharp because even if 
$(p+1)F_i F_{i+2}-F_i-F_{i+2}>F_{2 i+h}$, it is possible to have $g_p(F_i,F_{i+2},F_{i+k})=g_p(F_i,F_{i+2})$ ($k\ge i+h$).

\section{Lucas numbers}  

The formulas about Fibonacci numbers can be applied to obtain those about Lucas numbers. The discussion is similar, though the value $\fl{(L_i-1)/F_k}$ is different from $\fl{(F_i-1)/F_k}$. So, we list the results only.

When $p=1$, we have the following.  

\begin{theorem}  
For $i\ge 3$, we have 
\begin{align*}  
g_1(L_i,L_{i+2},L_{i+k})&=(2 L_i-1)L_{i+2}-L_i\quad(k\ge i+4)\,, 
%\label{eq:l1+4}
\\
g_1(L_i,L_{i+2},L_{2 i+3})&=(F_{i+3}-1)L_{i+2}-L_i\,, 
%\label{eq:l1+3}
\\
g_1(L_i,L_{i+2},L_{2 i+2})&=(3 F_{i-1}-1)L_{i+2}+L_{2 i+2}-L_i\,. 
%\label{eq:l1+2}
%\\
%g_1(L_i,L_{i+2},L_{2 i+1})&=(F_{i-1}-1)L_{i+2}+2 L_{2 i+1}-L_i\,, 
%\label{eq:l1+1}\\
%g_1(L_i,L_{i+2},L_{2 i})&=\begin{cases}
%(F_{i-3}-1)L_{i+2}+3 L_{2 i}-L_i&\text{$(i\ge 4)$}\\
%L_{i+2}+2 L_{2 i}-L_i(=43)&\text{$(i=3)$}\,,
%\end{cases}
%\label{eq:l1+0}\\
%g_1(L_i,L_{i+2},L_{2 i-1})&=(F_{i-2}-1)L_{i+2}+4 L_{2 i-1}-L_i\quad(i\ge 4)\,, 
%\label{eq:l1-1}\\
%g_1(L_i,L_{i+2},L_{2 i-2})&=(L_{i-4}-1)L_{i+2}+6 L_{2 i-2}-L_i\quad(i\ge 5)\,. 
%\label{eq:l1-2}
\end{align*} 
When $r=\fl{(L_i-1)/F_k}\ge 1$, that is, $k\le i+1$, we have 
\begin{align*}  
&g_1(L_i,L_{i+2},L_{i+k})\\
&=\begin{cases}
(L_i-r F_k-1)L_{i+2}+(r+1)L_{i+k}-L_i&\text{if $(L_i-r F_k)L_{i+2}\ge F_{k-2}L_i$},\\
(F_k-1)L_{i+2}+r L_{i+k}-L_i&\text{if $(L_i-r F_k)L_{i+2}<F_{k-2}L_i$}. 
\end{cases}
%\label{eq:d1-123}  
\end{align*}
\label{th:l1} 
\end{theorem}

When $p=2$, we have the following.  

\begin{theorem}  
For $i\ge 3$, we have 
\begin{align*}  
g_2(L_i,L_{i+2},L_{i+k})&=(3 L_i-1)L_{i+2}-L_i\quad(k\ge i+4)\,, 
%\label{eq:l2+4}
\\
g_2(L_i,L_{i+2},L_{2 i+3})&=(L_{i}-1)L_{i+2}+L_{2 i+3}-L_i\,, 
%\label{eq:l2+3}
\\
g_2(L_i,L_{i+2},L_{2 i+2})&=\begin{cases}
(L_{i}-1)L_{i+2}+L_{2 i+2}-L_i&\text{$(i$ is odd$)$}\\
(2 L_{i}-1)L_{i+2}-L_i&\text{$(i$ is even$)$}\,,
\end{cases}
%\label{eq:l2+2}
\\
g_2(L_i,L_{i+2},L_{2 i+1})&=(2 F_{i-1}-1)L_{i+2}+2 L_{2 i+1}-L_i\,,  
%\label{eq:l2+1} 
\\ 
g_2(L_i,L_{i+2},L_{2 i})&=L_{i+2}+3 L_{2 i}-L_i(=61)\quad(i=3)\,.
%\\
%g_2(L_i,L_{i+2},L_{2 i})&=\begin{cases}
%(F_{i-3}-1)L_{i+2}+4 L_{2 i}-L_i&\text{$(i\ge 4)$}\\
%L_{i+2}+3 L_{2 i}-L_i(=61)&\text{$(i=3)$}\,,
%\end{cases}
%\label{eq:l2+0}\\
%g_2(L_i,L_{i+2},L_{2 i-1})&=(F_{i-2}-1)L_{i+2}+5 L_{2 i-1}-L_i\quad(i\ge 4)\,,
%\label{eq:l2-1}\\
%g_2(L_i,L_{i+2},L_{2 i-2})&=(L_{i-4}-1)L_{i+2}+7 L_{2 i-2}-L_i\quad(i\ge 5)\,.
%\label{eq:l2-2}
\end{align*} 
When $r=\fl{(L_i-1)/F_k}\ge 2$, that is, $k\le i$ except $i=k=3$, we have 
\begin{align*}  
&g_2(L_i,L_{i+2},L_{i+k})\\
&=\begin{cases}
(L_i-r F_k-1)L_{i+2}+(r+2)L_{i+k}-L_i&\text{if $(L_i-r F_k)L_{i+2}\ge F_{k-2}L_i$},\\
(F_k-1)L_{i+2}+(r+1)L_{i+k}-L_i&\text{if $(L_i-r F_k)L_{i+2}<F_{k-2}L_i$}. 
\end{cases}
%\label{eq:d2-123}  
\end{align*}

\label{th:l2} 
\end{theorem}

When $p=3$, we have the following.  

\begin{theorem}  
For $i\ge 3$, we have 
\begin{align*}  
g_3(L_i,L_{i+2},L_{i+k})&=(4 L_i-1)L_{i+2}-L_i\quad(k\ge i+5)\,, 
%\label{eq:l3+5}
\\
g_3(L_i,L_{i+2},L_{2 i+4})&=(4 F_{i-1}-F_{i-2}-1)L_{i+2}+L_{2 i+4}-L_i\,, 
%\label{eq:l3+4}
\\
g_3(L_i,L_{i+2},L_{2 i+3})&=(4 F_{i+1}-1)L_{i+2}-L_i\,, 
%\label{eq:l3+3}
\\
g_3(L_i,L_{i+2},L_{2 i+2})&=(F_{i}+2 F_{i-3}-1)L_{i+2}+2 L_{2 i+2}-L_i\,, 
%\label{eq:l3+2}
\\
g_3(L_i,L_{i+2},L_{2 i+1})&=(F_{i-1}-1)L_{i+2}+3 L_{2 i+1}-L_i\,, 
%\label{eq:l3+1}
\\
g_3(L_i,L_{i+2},L_{2 i})&=\begin{cases}
(2 F_{i-3}-1)L_{i+2}+4 L_{2 i}-L_i&\text{$(i\ge 4)$}\\
3 L_{i+2}+2 L_{2 i}-L_i(=69)&\text{$(i=3)$}\,. 
\end{cases}
%\label{eq:l3+0}
%\\
%g_3(L_i,L_{i+2},L_{2 i-1})&=(F_{i-2}-1)L_{i+2}+6 L_{2 i-1}-L_i\quad(i\ge 4)\,. 
%\label{eq:l3-1}
\end{align*} 
When $r=\fl{(L_i-1)/F_k}\ge 3$, that is, $k\le i-1$, we have 
\begin{align*}  
&g_3(L_i,L_{i+2},L_{i+k})\\
&=\begin{cases}
(L_i-r F_k-1)L_{i+2}+(r+3)L_{i+k}-L_i&\text{if $(L_i-r F_k)L_{i+2}\ge F_{k-2}L_i$},\\
(F_k-1)L_{i+2}+(r+2)L_{i+k}-L_i&\text{if $(L_i-r F_k)L_{i+2}<F_{k-2}L_i$}. 
\end{cases}
%\label{eq:d3-123}  
\end{align*}
\label{th:l3} 
\end{theorem}

For general $p$, when $r$ is not less than $p$, we have an explicit formula. 

\begin{theorem}  
Let $i\ge 3$ and $p$ be a nonnegative integer.   
When $r=\fl{(L_i-1)/F_k}\ge p$ with $(r,p)\ne(0,0)$, we have 
\begin{align*} 
&g_p(L_i,L_{i+2},L_{i+k})\notag\\
&=\begin{cases}
(L_i-r F_k-1)L_{i+2}+(r+p)L_{i+k}-L_i&\text{if $(L_i-r F_k)L_{i+2}\ge F_{k-2}L_i$};\\
(F_k-1)L_{i+2}+(r+p-1)L_{i+k}-L_i&\text{if $(L_i-r F_k)L_{i+2}<F_{k-2}L_i$}\,. 
\end{cases}
%\label{eq:llp-ppp}
\end{align*}
\label{th:llp} 
\end{theorem}

\section{The number of representations}

By using the table of complete residue systems, we can also find explicit formulas of the $p$-Sylvester number, which is the total number of nonnegative integers that can only be expressed in at most $p$ ways. When $p=0$, such a number is often called the Sylvester number.

\subsection{Main results when $p=1$}  

When $p=1$, we have the following.  

\begin{theorem}  
For $i\ge 3$, we have 
\begin{align}  
n_1(F_i,F_{i+2},F_{i+k})&=\frac{1}{2}(3 F_i F_{i+2}-F_i-F_{i+2}+1)\quad(k\ge i+2)\,, 
\label{eq:n1+2}\\
n_1(F_i,F_{i+2},F_{2 i+1})&=\frac{1}{2}(3 F_i F_{i+2}-F_i-F_{i+2}+1)-(2 F_i-F_k)F_{k-2}\,,\notag\\
&\qquad\qquad\qquad\qquad\qquad(k=i,i+1)\,.  
\label{eq:n1+10}
\end{align} 
When $r=\fl{(F_i-1)/F_k}\ge 1$, that is, $k\le i-1$, we have 
\begin{align}
&n_1(F_i,F_{i+2},F_{i+k})\notag\\
&=\frac{1}{2}\bigl(F_i+2 F_k-1)F_{i+2}-F_i+1\bigr)-\left(r F_i-\frac{(r-1)(r+2)}{2}F_k\right)F_{k-2}\,.
\label{eq:n1-123}  
\end{align}
\label{th:n1} 
\end{theorem}

\begin{proof} 
When $r=0$ and $2\ell+1\le F_k-1$, by $F_i-1=\ell$, 
\begin{align*}  
\sum_{j=0}^{F_i-1}m_j^{(1)}&=t_{\ell+1,0}+\cdots+t_{2\ell+1,0}\\
&=\left(\frac{(2\ell+1)(2\ell+2)}{2}-\frac{\ell(\ell+1)}{2}\right)F_{i+2}\\
&=\frac{(3 F_i-1)F_i F_{i+2}}{2}\,.
\end{align*}
Hence, by Lemma \ref{cor-mp} (\ref{mp-n}), we have 
\begin{align*} 
n_1(F_i,F_{i+2},F_{i+k})&=\frac{(3 F_i-1)F_{i+2}}{2}-\frac{F_i-1}{2}\\
&=\frac{1}{2}(3 F_i F_{i+2}-F_i-F_{i+2}+1)\,, 
\end{align*}
which is (\ref{eq:n1+2}). 

When $r=0$ and $2\ell+1\ge F_k$, by $F_i-1=\ell$, 
\begin{align*}  
\sum_{j=0}^{F_i-1}m_j^{(1)}&=(t_{\ell+1,0}+\cdots+t_{F_k-1,0})+(t_{0,1}+\cdots+t_{2\ell+1-F_k,1})\\
&=\left(\frac{(F_k-1)F_k}{2}-\frac{\ell(\ell+1)}{2}\right)F_{i+2}\\
&\quad +\frac{(2\ell+1-F_k)(2\ell+2-F_k)}{2}F_{i+2}+(2\ell+2-F_k)F_{i+k}\\
&=\left(\frac{(3 F_i-1)F_i}{2}-2 F_i F_k+F_k^2\right)F_{i+2}+(2 F_i-F_k)F_{i+k}\,.
\end{align*}
Since $F_{i+k}=F_{i+2}F_k-F_i F_{k-2}$, 
$$
\sum_{j=0}^{F_i-1}m_j^{(1)}=\left(\frac{(3 F_i-1)F_{i+2}}{2}-(2 F_i-F_k)F_{k-2}\right)F_i\,.
$$ 
Hence, by Lemma \ref{cor-mp} (\ref{mp-n}), we have 
\begin{align*} 
n_1(F_i,F_{i+2},F_{i+k})&=\frac{(3 F_i-1)F_{i+2}}{2}-(2 F_i-F_k)F_{k-2}-\frac{F_i-1}{2}\\
&=\frac{1}{2}(3 F_i F_{i+2}-F_i-F_{i+2}+1)-(2 F_i-F_k)F_{k-2}\,, 
\end{align*}
which is (\ref{eq:n1+10}). 

When $r\ge 1$, by $F_i-1=r F_k+\ell$, we have 
\begin{align*}  
&\sum_{j=0}^{F_i-1}m_j^{(1)}\\
&=\sum_{h=0}^{r-2}(t_{F_k,h}+\cdots+t_{2 F_k-1,h})
+(t_{F_k,r-1}+\cdots+t_{F_k+\ell,r-1})\\
&\quad +(t_{\ell+1,r}+\cdots+t_{F_k-1,r})+(t_{0,r+1}+\cdots+t_{\ell,r+1})\\
&=(r-1)\left(\frac{(2 F_k-1)(2 F_k)}{2}-\frac{(F_k-1)F_k}{2}\right)F_{i+2}+\frac{(r-2)(r-1)}{2}F_k F_{i+k}\\
&\quad +\left(\frac{(F_k+\ell)(F_k+\ell+1)}{2}-\frac{(F_k-1)F_k}{2}\right)F_{i+2}+(\ell+1)(r-1)F_{i+k}\\
&\quad +\left(\frac{(F_k-1)F_k}{2}-\frac{\ell(\ell+1)}{2}\right)F_{i+2}+(F_k-1-\ell)r F_{i+k}\\
&\quad +\frac{\ell(\ell+1)}{2}F_{i+2}+(\ell+1)(r+1)F_{i+k}\\
&=\frac{1}{2}\left((r-1)(3 F_k-1)F_k+\bigl(F_i-(r-1)F_k-1\bigr)\bigl(F_i-(r-1)F_k\bigr)\right)F_{i+2}\\
&\quad +\left(\frac{(r-2)(r-1)}{2}F_k+r\bigl(F_i-(r-1)F_k\bigr)\right)F_{i+k}\\
&=\frac{1}{2}(F_i+2 F_k-1)F_i F_{i+2}\\
&\quad -\left(r F_i-\frac{(r-1)(r+2)}{2}F_k\bigr)\right)F_{k-2}F_i\,. 
\end{align*}
Hence, by Lemma \ref{cor-mp} (\ref{mp-n}), we have 
\begin{align*} 
&n_1(F_i,F_{i+2},F_{i+k})\\
&=\frac{1}{2}(F_i+2 F_k-1)F_{i+2}-\left(r F_i-\frac{(r-1)(r+2)}{2}F_k\right)F_{k-2}-\frac{F_i-1}{2}\\
&=\frac{1}{2}\bigl((F_i+2 F_k-1)F_{i+2}-F_i+1\bigr)-\left(r F_i-\frac{(r-1)(r+2)}{2}F_k\right)F_{k-2}\,, 
\end{align*}
which is (\ref{eq:n1-123}). 
\end{proof}

\subsection{The case $p=2$}

When $p=2$, we have the following.  

\begin{theorem}  
For $i\ge 3$, we have 
\begin{align}  
n_2(F_i,F_{i+2},F_{i+k})&=\frac{1}{2}(5 F_i F_{i+2}-F_i-F_{i+2}+1)\quad(k\ge i+3)\,, 
\label{eq:n2+3}\\
n_2(F_i,F_{i+2},F_{2 i+2})&=\frac{1}{2}\bigl((7 F_{i+2}-6 F_i-1)F_{i}-F_{i+2}+1\bigr)\,,
\label{eq:n2+2}\\
n_2(F_i,F_{i+2},F_{2 i+1})&=\frac{1}{2}\bigl((7 F_{i+2}-8 F_i-1)F_{i}-F_{i+2}+1\bigr)\,, 
\label{eq:n2+1}\\
n_2(F_i,F_{i+2},F_{2 i})&=\frac{1}{2}(3 F_i F_{i+2}-F_i-F_{i+2}+1)\,, 
\label{eq:n2+0}\\
n_2(F_i,F_{i+2},F_{2 i-1})&=\frac{1}{2}\bigl((170 F_i-1)F_i+(24 F_{i+2}-125 F_i-1)F_{i+2}+1\bigr)\,. 
\label{eq:n2-1}
\end{align} 
When $r=\fl{(F_i-1)/F_k}\ge 2$, that is, $k\le i-2$, we have 
\begin{align}
&n_2(F_i,F_{i+2},F_{i+k})\notag\\
&=\frac{1}{2}\bigl((F_i+4 F_k-1)F_{i+2}-F_i+1\bigr)-\frac{1}{2}\bigl(2 r F_i-(r+3)(r-2)F_k\bigr)F_{k-2}\,. 
\label{eq:n2-234}
\end{align} 
\label{th:n2} 
\end{theorem}   

\begin{proof}
When $r=0$ and $F_k\ge 3\ell+3$, by $F_i-1=\ell$, we have 
\begin{align*}  
\sum_{j=0}^{F_i-1}m_j^{(2)}
&=t_{2\ell+2,0}+\cdots+t_{3\ell+2,0}\\
&=\left(\frac{(3\ell+2)(3\ell+3)}{2}-\frac{(2\ell+1)(2\ell+2)}{2}\right)F_{i+2}\\ 
&=\frac{1}{2}(5 F_i-1)F_i F_{i+2}\,. 
\end{align*} 
Hence, by Lemma \ref{cor-mp} (\ref{mp-n}), we have 
\begin{align*} 
n_2(F_i,F_{i+2},F_{i+k})&=\frac{1}{2}(5 F_i-1)F_{i+2}-\frac{F_i-1}{2}\\
&=\frac{1}{2}(5 F_i F_{i+2}-F_i-F_{i+2}+1)\,, 
\end{align*}  
which is (\ref{eq:n2+3}). 

When $r=0$ and $2\ell+2\le F_k\le 3\ell+2$, by $F_{i+k}=F_{i+2}F_k-F_i F_{k-2}$, we have 
\begin{align*}  
&\sum_{j=0}^{F_i-1}m_j^{(2)}\\
&=(t_{2\ell+2,0}+\cdots+t_{F_k-1,0})+(t_{0,1}+\cdots+t_{3\ell+2-F_k,1})\\
&=\left(\frac{(F_k-1)F_k}{2}-\frac{(2\ell+1)(2\ell+2)}{2}\right)F_{i+2}\\ 
&\quad +\frac{(3\ell+2-F_k)(3\ell+3-F_k)}{2}F_{i+2}+(3\ell+3-F_k)F_{i+k}\\
&=\frac{1}{2}(5 F_i-1)F_i F_{i+2}-(3 F_i-F_k)F_i F_{k-2}\,. 
\end{align*}
Hence, we have 
\begin{align*} 
n_2(F_i,F_{i+2},F_{i+k})&=\frac{1}{2}(5 F_i-1)F_{i+2}-(3 F_i-F_k)F_{k-2}-\frac{F_i-1}{2}\\
&=\frac{1}{2}(5 F_i F_{i+2}-F_i-F_{i+2}+1)-(3 F_i-F_k)F_{k-2}\,. 
\end{align*}  
This case occurs only when $k=i+2$. Hence, we get (\ref{eq:n2+2}). 

When $r=0$ and $F_k\le 2\ell+1$, we have 
\begin{align*}  
&\sum_{j=0}^{F_i-1}m_j^{(2)}
=(t_{F_k,0}+\cdots+t_{2\ell+1,0})+(t_{2\ell+2-F_k,1}+\cdots+t_{\ell,1})\\
&=\left(\frac{(2\ell+1)(2\ell+2)}{2}-\frac{(F_k-1)F_k}{2}\right)F_{i+2}\\ 
&\quad +\left(\frac{\ell(\ell+1)}{2}-\frac{(2\ell+1-F_k)(2\ell+2-F_k)}{2}\right)F_{i+2}+(F_k-\ell-1)F_{i+k}\\
&=\frac{1}{2}(F_i+2 F_k-1)F_i F_{i+2}-(F_k-F_i)F_i F_{k-2}\,. 
\end{align*}
Hence, we have 
\begin{align*} 
n_2(F_i,F_{i+2},F_{i+k})&=\frac{1}{2}(F_i+2 F_k-1)F_{i+2}-(F_k-F_i)F_{k-2}-\frac{F_i-1}{2}\\
&=\frac{1}{2}\bigl((F_i+2 F_k-1)F_{i+2}-F_i+1\bigr)-(F_k-F_i)F_{k-2}\,. 
\end{align*} 
This case occurs only when $k=i+1$. Hence, by rewriting we get (\ref{eq:n2+1}). 

When $r=1$ and $2\ell+2\le F_k$, by $F_i-1=F_k+\ell$, we have 
\begin{align*}  
&\sum_{j=0}^{F_i-1}m_j^{(2)}\\
&=(t_{F_k+\ell+1,0}+\cdots+t_{2 F_k-1,0})+(t_{F_k,1}+\cdots+t_{F_k+\ell,1})\\
&\quad +(t_{\ell+1,2}+\cdots+t_{2\ell+1,2})\\ 
&=\left(\frac{(2 F_k-1)(2 F_k)}{2}-\frac{(F_k+\ell)(F_k+\ell+1)}{2}\right)F_{i+2}\\ 
&\quad +\left(\frac{(F_k+\ell)(F_k+\ell+1)}{2}-\frac{(F_k-1)F_k}{2}\right)F_{i+2}+(\ell+1)F_{i+k}\\
&\quad +\left(\frac{(2\ell+1)(2\ell+2)}{2}-\frac{\ell(\ell+1)}{2}\right)F_{i+2}+2(\ell+1)F_{i+k}\\
&=\frac{1}{2}(3 F_i-1)F_i F_{i+2}-3(F_i-F_k)F_i F_{k-2}\,. 
\end{align*}
Hence, we have 
\begin{align*} 
n_2(F_i,F_{i+2},F_{i+k})&=\frac{1}{2}(3 F_i-1)F_{i+2}-3(F_i-F_k)F_{k-2}-\frac{F_i-1}{2}\\
&=\frac{1}{2}\bigl(3 F_i F_{i+2}-F_i-F_{i+2}+1\bigr)-3(F_i-F_k)F_{k-2}\,. 
\end{align*} 
This case occurs only when $k=i$. Hence, we get (\ref{eq:n2+0}). 

When $r=1$ and $2\ell+1\ge F_k$, we have 
\begin{align*}  
&\sum_{j=0}^{F_i-1}m_j^{(2)}\\
&=(t_{F_k+\ell+1,0}+\cdots+t_{2 F_k-1,0})+(t_{F_k,1}+\cdots+t_{F_k+\ell,1})\\
&\quad +(t_{\ell+1,2}+\cdots+t_{F_k-1,2})+(t_{0,3}+\cdots+t_{2\ell+1-F_k,3})\\ 
&=\left(\frac{(2 F_k-1)(2 F_k)}{2}-\frac{(F_k+\ell)(F_k+\ell+1)}{2}\right)F_{i+2}\\ 
&\quad +\left(\frac{(F_k+\ell)(F_k+\ell+1)}{2}-\frac{(F_k-1)F_k}{2}\right)F_{i+2}+(\ell+1)F_{i+k}\\
&\quad +\left(\frac{(F_k-1)F_k}{2}-\frac{\ell(\ell+1)}{2}\right)F_{i+2}+2(F_k-\ell-1)F_{i+k}\\
&\quad +\frac{(2\ell+1-F_k)(2\ell+2-F_k)}{2}F_{i+2}+3(2\ell+2-F_k)F_{i+k}\\
&=\frac{1}{2}(3 F_i-1)F_i F_{i+2}-(5 F_i-6 F_k)F_i F_{k-2}\,. 
\end{align*}
Hence, we have 
\begin{align*} 
n_2(F_i,F_{i+2},F_{i+k})&=\frac{1}{2}(3 F_i-1)F_{i+2}-(5 F_i-6 F_k)F_{k-2}-\frac{F_i-1}{2}\\
&=\frac{1}{2}\bigl(3 F_i F_{i+2}-F_i-F_{i+2}+1\bigr)-(5 F_i-6 F_k)F_{k-2}\,. 
\end{align*} 
This case occurs only when $k=i-1$. Hence, after rewriting, we get (\ref{eq:n2-1}).

When $r\ge 2$, by $F_i-1=r F_k+\ell$, we have 
\begin{align*}  
&\sum_{j=0}^{F_i-1}m_j^{(2)}\\
&=\sum_{h=0}^{r-3}(t_{2 F_k,h}+\cdots+t_{3 F_k-1,h})
+(t_{2 F_k,r-2}+\cdots+t_{2 F_k+\ell,r-2})\\
&\quad +(t_{F_k+\ell+1,r-1}+\cdots+t_{2 F_k-1,r-1})+(t_{F_k,r}+\cdots+t_{F_k+\ell,r})\\
&\quad +(t_{\ell+1,r+1}+\cdots+t_{F_k-1,r+1})+(t_{0,r+2}+\cdots+t_{\ell,r+2})\\
&=(r-2)\left(\frac{(3 F_k-1)(3 F_k)}{2}-\frac{(2 F_k-1)(2 F_k)}{2}\right)F_{i+2}+\frac{(r-3)(r-2)}{2}F_k F_{i+k}\\
&\quad +\left(\frac{(2 F_k+\ell)(2 F_k+\ell+1)}{2}-\frac{(2 F_k-1)(2 F_k)}{2}\right)F_{i+2}+(\ell+1)(r-2)F_{i+k}\\
&\quad +\left(\frac{(2 F_k-1)(2 F_k)}{2}-\frac{(F_k+\ell)(F_k+\ell+1)}{2}\right)F_{i+2}+(F_k-\ell-1)(r-1)F_{i+k}\\
&\quad +\left(\frac{(F_k+\ell)(F_k+\ell+1)}{2}-\frac{(F_k-1)F_k}{2}\right)F_{i+2}+(\ell+1)r F_{i+k}\\
&\quad +\left(\frac{(F_k-1)F_k}{2}-\frac{\ell(\ell+1)}{2}\right)F_{i+2}+(F_k-\ell-1)(r+1)F_{i+k}\\
&\quad +\frac{\ell(\ell+1)}{2}F_{i+2}+(\ell+1)(r+2)F_{i+k}\\
&=\frac{1}{2}(F_i+4 F_k-1)F_i F_{i+2}-\frac{1}{2}\bigl(2 r F_i-(r+3)(r-2)F_k\bigr)F_i F_{k-2}\,. 
\end{align*}
Hence, by Lemma \ref{cor-mp} (\ref{mp-n}), we have 
\begin{align*} 
&n_2(F_i,F_{i+2},F_{i+k})\\
&=\frac{1}{2}(F_i+4 F_k-1)F_{i+2}-\frac{1}{2}\bigl(2 r F_i-(r+3)(r-2)F_k\bigr)F_{k-2}
-\frac{F_i-1}{2}\\
&=\frac{1}{2}\bigl((F_i+4 F_k-1)F_{i+2}-F_i+1\bigr)-\frac{1}{2}\bigl(2 r F_i-(r+3)(r-2)F_k\bigr)F_{k-2}\,, 
\end{align*} 
which is (\ref{eq:n2-234}). 
\end{proof}

\subsection{The case $p=3$} 

When $p=3$, we have the following. The process is similar, and the proof is omitted.  

\begin{theorem}  
For $i\ge 3$, we have 
\begin{align}  
n_3(F_i,F_{i+2},F_{i+k})&=\frac{1}{2}(7 F_i F_{i+2}-F_i-F_{i+2}+1)\quad(k\ge i+3)\,, 
%\label{eq:n3+3}
\notag\\
n_3(F_i,F_{i+2},F_{2 i+2})&=(F_{i}-1)F_{i+2}+F_{2 i+2}-F_i\,, 
%\label{eq:n3+2}
\notag\\
n_3(F_i,F_{i+2},F_{2 i+1})&=(F_{i}+F_{i-2}-1)F_{i+2}+F_{2 i+1}-F_i\,, 
%\label{eq:n3+1}
\notag\\
n_3(F_i,F_{i+2},F_{2 i})&=\frac{1}{2}\bigl((5 F_i-1)F_{i+2}-F_i+1\bigr)-2 F_{i}F_{i-2}\,, 
%\label{eq:n3+0}
\notag\\
n_3(F_i,F_{i+2},F_{2 i-1})&=\frac{1}{2}\bigl((F_i+4 F_{i-1}-1)F_{i+2}-F_i+1\bigr)-2 F_{i-1}F_{i-3}\quad(i\ge 4)\,, 
%\label{eq:n3-1}
\notag\\
n_3(F_i,F_{i+2},F_{2 i-2})&=\frac{1}{2}\bigl((3 F_i-1)F_{i+2}-F_i+1\bigr)-(8 F_i-15 F_{i-2})F_{i-4}\quad(i\ge 5)\,. 
\label{eq:n3-2}
\end{align} 
When $r=\fl{(F_i-1)/F_k}\ge 3$, that is, $k\le i-3$, we have 
\begin{align*} 
&n_3(F_i,F_{i+2},F_{i+k})\notag\\
&=\frac{1}{2}\bigl((F_i+6 F_k-1)F_{i+2}-F_i+1\bigr)-\frac{1}{2}\bigl(2 r F_i-(r+4)(r-3)F_k\bigr)F_{k-2}\,. 
%\label{eq:n3-345}
\end{align*}
\label{th:n3} 
\end{theorem}

\subsection{General $p$ case}  

We can continue to obtain explicit formulas of $n_p(F_i,F_{i+2},F_{i+k})$ for $p=4,5,\dots$. However, the situation becomes more complicated. We need more case-by-case discussions.

For general $p$, when $r\ge p$, we can have an explicit formula.  

\begin{theorem}  
Let $i\ge 3$ and $p$ be a nonnegative integer.   
When $r=\fl{(F_i-1)/F_k}\ge p$, we have 
\begin{align} 
&n_p(F_i,F_{i+2},F_{i+k})\notag\\
&=\frac{1}{2}\bigl((F_i+2 p F_k-1)F_{i+2}-F_i+1\bigr)-\frac{1}{2}\bigl(2 r F_i-(r+p+1)(r-p)F_k\bigr)F_{k-2}\,. 
\label{eq:np-ppp}
\end{align}
\label{th:np} 
\end{theorem} 

\noindent 
{\it Remark.}  
When $p=0$, Theorem \ref{th:np} reduces to \cite[Corollary 2]{MAR}.    

\begin{proof}[Sketch of the proof of Theorem \ref{th:np}] 
We have 
\begin{align*}  
&\sum_{j=0}^{F_i-1}m_j^{(p)}\\
&=\sum_{h=0}^{r-p-1}(t_{p F_k,h}+\cdots+t_{(p+1)F_k-1,h})\\
&\quad +(t_{p F_k,r-p}+\cdots+t_{p F_k+\ell,r-p})\\
&\quad +(t_{(p-1)F_k+\ell+1,r-p+1}+\cdots+t_{(p-1)F_k-1,r-p+1})\\
&\quad +(t_{(p-1)F_k,r-p+2}+\cdots+t_{(p-1)F_k+\ell,r-p+2})\\
&\quad +\cdots\\ 
&\quad +(t_{\ell+1,r+p-1}+\cdots+t_{F_k-1,r+p-1})+(t_{0,r+p}+\cdots+t_{\ell,r+p})\\
&=\frac{1}{2}\bigl((r-p)((2 p+1)F_k-1)F_k+(F_i-(r-p)F_k-1)(F_i-(r-p)F_k)\bigr)F_{i+2}\\
&\quad +\left(\frac{(r-p-1)(r-p)}{2}F_k+r(F_i-(r-p)F_k)\right)F_{i+k}\\
&=\frac{1}{2}(F_i+2 p F_k-1)F_i F_{i+2}-\frac{1}{2}\bigl(2 r F_i-(r+p+1)(r-p)F_k\bigr)F_i F_{k-2}\,. 
\end{align*}
Hence, by Lemma \ref{cor-mp} (\ref{mp-n}), we have 
\begin{align*} 
&n_p(F_i,F_{i+2},F_{i+k})\\
&=\frac{1}{2}(F_i+2 p F_k-1)F_{i+2}-\frac{1}{2}\bigl(2 r F_i-(r+p+1)(r-p)F_k\bigr)F_{k-2}
-\frac{F_i-1}{2}\\
&=\frac{1}{2}\bigl((F_i+2 p F_k-1)F_{i+2}-F_i+1\bigr)-\frac{1}{2}\bigl(2 r F_i-(r+p+1)(r-p)F_k\bigr)F_{k-2}\,,  
\end{align*} 
which is (\ref{eq:np-ppp}). 
\end{proof}

\section{Example}  

Consider the Fibonacci triple $(F_6,F_8,F_{10})$. Since $F_6-1=2 F_4+1$, we see that $r=2$ and $\ell=1$. Then, we can construct the first least set, the second least  and 3rd, 4th and 5th least sets of the complete residue systems as follows.  
\begin{align*}
\ctext{1}{\rm Ap} (F_6;0)&=\{0,21,42,55,76,97,110,131\}\pmod{F_6}\,,\\ 
\ctext{2}{\rm Ap} (F_6;1)&=\{63,84,105,118,139,152,165,186\}\pmod{F_6}\,,\\ 
\ctext{3}{\rm Ap} (F_6;2)&=\{126,147,160,173,194,207,220,241\}\pmod{F_6}\,,\\ 
\ctext{4}{\rm Ap} (F_6;3)&=\{168,181,202,215,228,249,262,275\}\pmod{F_6}\,,\\ 
\ctext{5}{\rm Ap} (F_6;4)&=\{189,210,223,236,257,270,283,296\}\pmod{F_6}\,.\\ 
\end{align*}

\begin{table}[htbp]
  \centering
\scalebox{0.7}{
\begin{tabular}{ccccccccccc}
&&\multicolumn{1}{c|}{}&&&\multicolumn{1}{c|}{}&&&\multicolumn{1}{c|}{}&&\\ 
\cline{1-2}\cline{3-4}\cline{5-6}\cline{7-8}\cline{9-10}\cline{11-11}
\multicolumn{1}{|c}{$0$}&$21$&\multicolumn{1}{c|}{$42$}&$63$&$84$&\multicolumn{1}{c|}{$105$}&$126$&\multicolumn{1}{c|}{$147$}&\multicolumn{1}{c|}{$168$}&$189$&\multicolumn{1}{c|}{$210$}\\ 
\cline{6-7}\cline{8-9}\cline{10-11}
\multicolumn{1}{|c}{$55$}&$76$&\multicolumn{1}{c|}{$97$}&$118$&\multicolumn{1}{c|}{$139$}&\multicolumn{1}{c|}{$160$}&$181$&\multicolumn{1}{c|}{$202$}&\multicolumn{1}{c|}{$223$}&&\\ 
\cline{3-4}\cline{5-6}\cline{7-8}\cline{9-9} 
\multicolumn{1}{|c}{$110$}&\multicolumn{1}{c|}{$131$}&\multicolumn{1}{c|}{$152$}&$173$&\multicolumn{1}{c|}{$194$}&\multicolumn{1}{c|}{$215$}&$236$&\multicolumn{1}{c|}{$257$}&\multicolumn{1}{c|}{}&&\\ 
\cline{1-2}\cline{3-4}\cline{5-6}\cline{7-8}
\multicolumn{1}{|c}{$165$}&\multicolumn{1}{c|}{$186$}&\multicolumn{1}{c|}{$207$}&$228$&\multicolumn{1}{c|}{$249$}&\multicolumn{1}{c|}{$270$}&&&\multicolumn{1}{c|}{}&&\\ 
\cline{1-2}\cline{3-4}\cline{5-6}
\multicolumn{1}{|c}{$220$}&\multicolumn{1}{c|}{$241$}&\multicolumn{1}{c|}{$262$}&\multicolumn{1}{c|}{$283$}&&\multicolumn{1}{c|}{}&&&\multicolumn{1}{c|}{}&&\\ 
\cline{1-2}\cline{3-4}
\multicolumn{1}{|c|}{$275$}&\multicolumn{1}{c|}{$296$}&\multicolumn{1}{c|}{}&&&\multicolumn{1}{c|}{}&&&\multicolumn{1}{c|}{}&&\\
\cline{1-2}
\end{tabular}
} 
%  \caption{${\rm Ap}(F_i;p)$ ($p=0,1,2$) for $r=1$ and $F_k\le 2\ell+1$}
%  \label{tb:g2system211}
\end{table}

\begin{table}[htbp]
  \centering
\scalebox{0.7}{
\begin{tabular}{ccccccccccc}
&&\multicolumn{1}{c|}{}&&&\multicolumn{1}{c|}{}&&&\multicolumn{1}{c|}{}&&\\ 
\cline{1-2}\cline{3-4}\cline{5-6}\cline{7-8}\cline{9-10}\cline{11-11}
\multicolumn{1}{|c}{$0$}&$5$&\multicolumn{1}{c|}{$2$}&$7$&$4$&\multicolumn{1}{c|}{$1$}&$6$&\multicolumn{1}{c|}{$3$}&\multicolumn{1}{c|}{$0$}&$5$&\multicolumn{1}{c|}{$2$}\\ 
\cline{6-7}\cline{8-9}\cline{10-11}
\multicolumn{1}{|c}{$7$}&$4$&\multicolumn{1}{c|}{$1$}&$6$&\multicolumn{1}{c|}{$3$}&\multicolumn{1}{c|}{$0$}&$5$&\multicolumn{1}{c|}{$2$}&\multicolumn{1}{c|}{$7$}&&\\ 
\cline{3-4}\cline{5-6}\cline{7-8}\cline{9-9} 
\multicolumn{1}{|c}{$6$}&\multicolumn{1}{c|}{$3$}&\multicolumn{1}{c|}{$0$}&$5$&\multicolumn{1}{c|}{$2$}&\multicolumn{1}{c|}{$7$}&$4$&\multicolumn{1}{c|}{$1$}&\multicolumn{1}{c|}{}&&\\ 
\cline{1-2}\cline{3-4}\cline{5-6}\cline{7-8}
\multicolumn{1}{|c}{$5$}&\multicolumn{1}{c|}{$2$}&\multicolumn{1}{c|}{$7$}&$4$&\multicolumn{1}{c|}{$1$}&\multicolumn{1}{c|}{$6$}&&&\multicolumn{1}{c|}{}&&\\ 
\cline{1-2}\cline{3-4}\cline{5-6}
\multicolumn{1}{|c}{$4$}&\multicolumn{1}{c|}{$1$}&\multicolumn{1}{c|}{$6$}&\multicolumn{1}{c|}{$3$}&&\multicolumn{1}{c|}{}&&&\multicolumn{1}{c|}{}&&\\ 
\cline{1-2}\cline{3-4}
\multicolumn{1}{|c|}{$3$}&\multicolumn{1}{c|}{$0$}&\multicolumn{1}{c|}{}&&&\multicolumn{1}{c|}{}&&&\multicolumn{1}{c|}{}&&\\
\cline{1-2}
\end{tabular}  \qquad 
\begin{tabular}{ccccccccccc}
&&\multicolumn{1}{c|}{}&&&\multicolumn{1}{c|}{}&&&\multicolumn{1}{c|}{}&&\\ 
\cline{1-2}\cline{3-4}\cline{5-6}\cline{7-8}\cline{9-10}\cline{11-11}
\multicolumn{1}{|c}{}&&\multicolumn{1}{c|}{}&&$\ctext{2}$&\multicolumn{1}{c|}{}&$\ctext{3}$&\multicolumn{1}{c|}{}&\multicolumn{1}{c|}{$\ctext{4}$}&$\ctext{5}$&\multicolumn{1}{c|}{}\\ 
\cline{6-7}\cline{8-9}\cline{10-11}
\multicolumn{1}{|c}{}&$\ctext{1}$&\multicolumn{1}{c|}{}&&\multicolumn{1}{c|}{}&\multicolumn{1}{c|}{$\ctext{3}$}&$\ctext{4}$&\multicolumn{1}{c|}{}&\multicolumn{1}{c|}{$\ctext{5}$}&&\\ 
\cline{3-4}\cline{5-6}\cline{7-8}\cline{9-9} 
\multicolumn{1}{|c}{}&\multicolumn{1}{c|}{}&\multicolumn{1}{c|}{$\ctext{2}$}&$\ctext{3}$&\multicolumn{1}{c|}{}&\multicolumn{1}{c|}{$\ctext{4}$}&$\ctext{5}$&\multicolumn{1}{c|}{}&\multicolumn{1}{c|}{}&&\\ 
\cline{1-2}\cline{3-4}\cline{5-6}\cline{7-8}
\multicolumn{1}{|c}{$\ctext{2}$}&\multicolumn{1}{c|}{}&\multicolumn{1}{c|}{$\ctext{3}$}&$\ctext{4}$&\multicolumn{1}{c|}{}&\multicolumn{1}{c|}{$\ctext{5}$}&&&\multicolumn{1}{c|}{}&&\\ 
\cline{1-2}\cline{3-4}\cline{5-6}
\multicolumn{1}{|c}{$\ctext{3}$}&\multicolumn{1}{c|}{}&\multicolumn{1}{c|}{$\ctext{4}$}&\multicolumn{1}{c|}{$\ctext{5}$}&&\multicolumn{1}{c|}{}&&&\multicolumn{1}{c|}{}&&\\ 
\cline{1-2}\cline{3-4}
\multicolumn{1}{|c|}{$\ctext{4}$}&\multicolumn{1}{c|}{$\ctext{5}$}&\multicolumn{1}{c|}{}&&&\multicolumn{1}{c|}{}&&&\multicolumn{1}{c|}{}&&\\
\cline{1-2}
\end{tabular}
} 
  \caption{${\rm Ap}(F_6,j)$ ($j=0,1,2,3,4$)}
%  \label{tb:g2system211}
\end{table}

Therefore, by Lemma \ref{cor-mp} (\ref{mp-g}) with (\ref{eq:bs}), we obtain that   
\begin{align*}  
g_0(F_6,F_8,F_{10})&=131-8=123\,,\\
g_1(F_6,F_8,F_{10})&=186-8=178\,,\\
g_2(F_6,F_8,F_{10})&=241-8=233\,,\\
g_3(F_6,F_8,F_{10})&=275-8=267\,,\\
g_4(F_6,F_8,F_{10})&=296-8=288\,.  
\end{align*}
By Lemma \ref{cor-mp} (\ref{mp-n}) with (\ref{eq:se}), we obtain that 
\begin{align*}  
n_0(F_6,F_8,F_{10})&=\frac{0+21+\cdots+131}{8}-\frac{8-1}{2}=63\,,\\
n_1(F_6,F_8,F_{10})&=\frac{63+84+\cdots+186}{8}-\frac{8-1}{2}=123\,,\\
n_2(F_6,F_8,F_{10})&=\frac{126+147+\cdots+241}{8}-\frac{8-1}{2}=180\,,\\
n_3(F_6,F_8,F_{10})&=\frac{168+181+\cdots+275}{8}-\frac{8-1}{2}=219\,,\\
n_4(F_6,F_8,F_{10})&=\frac{189+210+\cdots+296}{8}-\frac{8-1}{2}=242\,,  
\end{align*} 

On the other hand, 
from (\ref{eq:d1-123}), by $(F_6-2 F_4)F_8>F_2 F_6$, we get 
\begin{align*} 
g_1(F_6,F_8,F_{10})&=(F_6-2 F_4-1)F_{8}+3 F_{10}-F_6\\ 
&=178\,. 
\end{align*}
From (\ref{eq:d2-234}) and (\ref{eq:d3-2}), we get 
\begin{align*} 
g_2(F_6,F_8,F_{10})&=(F_6-2 F_4-1)F_{8}+4 F_{10}-F_6\\ 
&=233\,,\\
g_3(F_6,F_8,F_{10})&=(F_{1}-1)F_{8}+5 F_{10}-F_6\\ 
&=267\,, 
\end{align*}
respectively.  
From (\ref{eq:n1-123}), (\ref{eq:n2-234}) and (\ref{eq:n3-2}), we get 
\begin{align*} 
&n_1(F_6,F_8,F_{10})\\
&=\frac{1}{2}\bigl(F_6+2 F_4-1)F_{8}-F_6+1\bigr)-\left(2 F_6-\frac{(2-1)(2+2)}{2}F_4\right)F_{2}\\ 
&=123\,,\\ 
&n_2(F_6,F_8,F_{10})\\
&=\frac{1}{2}\bigl((F_6+4 F_4-1)F_{8}-F_6+1\bigr)-\frac{1}{2}\bigl(4 F_6-(2+3)(2-2)F_4\bigr)F_{2}\\
&=180\,,\\ 
&n_3(F_6,F_8,F_{10})\\
&=\frac{1}{2}\bigl((3 F_6-1)F_{8}-F_6+1\bigr)-(8 F_6-15 F_{4})F_{2}\\
&=219\,, 
\end{align*}
respectively.

\section{Open problems}  

In \cite{Fel2009}, a more general triple $g(F_a,F_b,F_c)$ is studied for distinct Fibonacci numbers with $a,b,c\ge 3$. 
In \cite{SKT}, the Frobenius number $g(a,a+b,2 a+3 b,\dots,F_{2k-1}a+F_{2 k}b)$ is given for relatively prime integers $a$ and $b$. 
Will we be able to say anything in terms of these $p$-Frobenius numbers?

\section*{Conflict of interests} 

There is no conflict of interests.

\section*{Acknowledgements}  

The authors thank the anonymous referees for careful reading of this manuscript.

\end{document}